\documentclass{article}
\usepackage[english]{babel}
\usepackage[utf8]{inputenc}
\usepackage[T1]{fontenc}
\usepackage{graphicx}
\usepackage{calrsfs}
\usepackage{enumerate}
\usepackage{enumitem}
\usepackage{cleveref}
\usepackage{amsmath,amssymb,amsfonts,amsthm}
\usepackage{comment}
\usepackage{color}
\usepackage[lofdepth,lotdepth]{subfig}
\usepackage{pgf,tikz}
\usepackage{pgfplots}
\usepackage{arxiv}

\usetikzlibrary{arrows}
      \title{Shape derivatives for the penalty formulation of contact problems with Tresca friction}

      \author{Bastien Chaudet-Dumas \\
      GIREF\thanks{Groupe Interdisciplinaire de Recherche en \'El\'ements Finis} \\
      D\'epartment de Math\'ematiques et Statistiques \\
      Universit\'e Laval \\
      Qu\'ebec, QC, Canada \\
      \texttt{bastien.chaudet.1@ulaval.ca} \\
         \And
      Jean Deteix \\   
      GIREF \\
      D\'epartment de Math\'ematiques et Statistiques \\
      Universit\'e Laval \\
      Qu\'ebec, QC, Canada \\
      \texttt{jean.deteix@mat.ulaval.ca} \\
     }

\DeclareMathAlphabet\mathpzc{T1}{pzc}{mb}{it}
\DeclareMathAlphabet{\pazocal}{OMS}{cmsy}{m}{n}

\def\norml{\left\|}
\def\normr{\right\|}
\DeclareMathOperator{\circl}{\circ\,}
\DeclareMathOperator{\divv}{div}
\DeclareMathOperator{\Divv}{\textbf{div}}
\DeclareMathOperator{\gradd}{\nabla\!}
\DeclareMathOperator{\grad}{\nabla\!}
\DeclareMathOperator{\maxx}{p_{+}}
\def\dmaxx{d\hspace{-0.1em}\maxx}
\def\dqq{d\mathbf{q}}

\DeclareMathOperator{\normalInt}{\mathbf{n}_\mathbf{o}}
\DeclareMathOperator{\tanInt}{\mathbf{t}_\mathbf{o}}
\DeclareMathOperator{\normalExt}{\mathbf{n}}
\DeclareMathOperator{\tanExt}{\mathbf{t}}

\DeclareMathOperator{\JacV}{J_\Omega}
\DeclareMathOperator{\JacB}{J_\Gamma}

\DeclareMathOperator{\hh}{\mathbf{h}}

\DeclareMathOperator{\pp}{\mathbf{p}}
\DeclareMathOperator{\qq}{\mathbf{q}}

\DeclareMathOperator{\uu}{\mathbf{u}}
\DeclareMathOperator{\vv}{\mathbf{v}}
\DeclareMathOperator{\ww}{\mathbf{w}}

\DeclareMathOperator{\zz}{\mathbf{z}}
\DeclareMathOperator{\gG}{\mathbf{g}}
\DeclareMathOperator{\Aa}{\mathbb{C}}
\DeclareMathOperator{\Kk}{\mathbf{K}}
\DeclareMathOperator{\Ll}{\mathbf{L}}
\DeclareMathOperator{\Hh}{\mathbf{H}}
\DeclareMathOperator{\Cc}{\pmb{\pazocal{C}}}
\DeclareMathOperator{\Ii}{\mathbf{I}}
\DeclareMathOperator{\Id}{\mathrm{Id}}
\DeclareMathOperator{\ff}{\mathbf{f}}
\DeclareMathOperator{\tauu}{\boldsymbol{\tau}}
\DeclareMathOperator{\epsilonn}{\boldsymbol{\epsilon}}
\DeclareMathOperator{\sigmaa}{\boldsymbol{\sigma}}

\DeclareMathOperator{\thetaa}{\boldsymbol{\theta}}

\DeclareMathOperator{\Xx}{\mathbf{X}}

\newtheorem{remark}{Remark}[section]
\newtheorem{theo}{Theorem}[section]
\newtheorem{cor}{Corollary}[section]
\newtheorem{lemma}{Lemma}[section]
\newtheorem{defn}{Definition}[section]
\newtheorem{hypothesis}{\sc Assumption}
\crefname{hypothesis}{Assumption}{Assumption}
\newtheorem*{notation}{Notation}

\crefname{theo}{Theorem}{Theorems}
\crefname{cor}{Corollary}{Corollaries}
\crefname{lemma}{Lemma}{Lemmas}
\crefname{defn}{Definition}{Definitions}

\definecolor{dartmouthgreen}{rgb}{0.05, 0.5, 0.06}

\usepackage{accents}

\usepackage{color}
\usepackage[normalem]{ulem}
\usepackage[colorinlistoftodos,textwidth=2.25cm]{todonotes}

\DeclareMathOperator{\Ww}{\mathbf{W}}
\def\prodL2#1#2{\left(#1\right)_{#2}}

\begin{document}

\maketitle 

\begin{abstract}
  In this article, the shape optimization of a linear elastic body subject to frictional (Tresca) contact is investigated. Due to the projection operators involved in the formulation of the contact problem, the solution is not shape differentiable in general. Moreover, shape optimization of the contact zone requires the computation of the gap between the bodies in contact, as well as its shape derivative. Working with directional derivatives, sufficient conditions for shape differentiability are derived. 
  Then, some numerical results, obtained with a gradient descent algorithm based on those shape derivatives, are presented.
\end{abstract}


\section{Introduction}

Optimal design is becoming a key element in industrial conception and applications. As the interest to include shape optimization in the design cycle of structures broadens, we are confronted with increasingly complex mechanical context. Large deformations, plasticity, contact and such can lead to difficult mathematical formulation. The non-linearities and/or non-differentiabilities stemming from the mechanical model give rise to complex shape sensitivity analysis which often requires a specific and delicate treatment.

This article deals with bodies in frictional (Tresca model) contact with a rigid foundation. Therefore this model is concerned with the non-penetrability and the eventual friction of the bodies in contact. From the mathematical point of view, it takes the form of an elliptic variational inequality of the second kind, see for example \cite{DuvLio1972,boieri1987existence} for existence, uniqueness, and regularity results.

Our approach to solve shape optimization problems is based on a gradient descent and Hadamard's boundary variation method, which requires the shape derivative of the cost functional. Such approaches, following the pioneer work \cite{hadamard1908memoire}, have been widely studied for the past forty years, for example in \cite{murat1975etude,simon1980differentiation,pironneau1982optimal,sokolowski1992introduction,DelZol2001,henrot2006variation}, to name a few. Obviously this raises the question of the differentiability of the cost functional with respect to the domain, which naturally leads to shape sensitivity analysis of the associated variational inequality. More specifically, as in \cite{allaire2004structural}, we use a level-set representation of the shapes, which allows to deal with changes of topology during the optimization process. Regarding more general topology optimization methods, let us mention \textit{density methods}, in which the shape is represented by a local density of material inside a given fixed domain. Among the most popular, we cite \cite{BenSig2003} for the \textit{SIMP method} (Solid Isotropic Material with Penalisation) and  \cite{allaire2012shape} for the \textit{homogenization method}.

As projection operators are involved in the formulation, the solution map is non-differentiable with respect to the shape or any other control parameter. There exist three main approaches to treat this non-differentiability.
The first one was introduced in \cite{mignot1984optimal}, where the author proves differentiability in a weaker sense, namely conical differentiability, and derives optimality conditions using this notion. We mention \cite{sokolowski1988shape} for the application of this method to shape sensitivity analysis of contact problems with Tresca friction. 
Another approach consists in discretizing the formulation, then use the tools from subdifferential calculus, see the series of papers \cite{kovccvara1994optimization,beremlijski2002shape,haslinger2012shape,beremlijski2014shape}, in the context of shape optimization for elastic bodies in frictional contact with a plane. 
The third approach, which is very popular in mechanical engineering, is to consider the penalized contact problem, which takes the form of a variational equality, then regularize all non-smooth functions. This leads to an approximate formulation having a Fréchet differentiable solution map. 
%
%
%
Following this penalty/regularization approach, we mention \cite{kim2000meshless} for two-dimensional parametric shape optimization, where the authors consider contact with a general rigid foundation and get interested in the differentiation of the gap. We also mention the more recent work \cite{maury2017shape} for shape optimization using the level set method (see \cite{allaire2004structural}), where the authors compute shape derivatives for the continuous problem in two and three dimensions, but do not take into account a possible gap between the bodies in contact. 
The same approach can be found in the context of optimal control, see among others \cite{ito2000optimal} for the general framework, and \cite{amassad2002optimal} for the specific case of frictional (Tresca) contact mechanics.

Let us finally mention the substantial work of Haslinger et al., who proved existence of optimal shapes for contact problems in some specific cases, see \cite{haslinger1985existence,haslinger1985shape}. Moreover, in  \cite{haslinger1986shape}, they proved consistency of the penalty approach in this context. 

In this paper, we aim at expressing shape derivatives for the continuous penalty formulation of frictional contact problems of Tresca type. Our approach is similar to the penalty/regularization, but we do not regularize non-smooth functions involved in the formulation. Indeed, shape differentiability does not require Fréchet-differentiability of the solution map, which makes the regularization step unnecessary.
Especially, the goal is to get similar results to \cite{maury2017shape} whitout regularizing, and extend those results in two ways.
First, we add a gap in the formulation, which enables to completely optimize the contact zone, as in \cite{kim2000meshless}. 
Second, we work in the slightly more general case where the Tresca threshold is not necessarily constant. 
This way, the formulae obtained could also be used in the context of the numerical approximation of a regularized Coulomb friction law by a fixed-point of Tresca problems. We refer to \cite{DuvLio1972,oden1983nonlocal,cocu1984existence} for existence and uniqueness results for this regularized Coulomb problem, and to \cite{hueber2008primal} (among others) for its numerical resolution by means of a fixed-point algorithm. 

This work is structured as follows. Section 2 presents the problem, its formulation and some related notations. Section 3 is dedicated to shape optimization. Especially, we express sufficient conditions for the solution of the penalty formulation to be shape differentiable (\cref{ThmExistMDer,CorExistSDer}). The shape optimization algorithm of gradient type, based on those shape derivatives, is briefly discussed. Finally, in section 5, some numerical results are exposed.

\section{Problem formulation}

\subsection{Geometrical setting}\label{sec:geo}
The body $\Omega \subset \mathbb{R}^d$, $d \in \{2,3\}$, is assumed to have $\pazocal{C}^1$ boundary, and to be in contact with a rigid foundation $\Omega_{rig}$, which has a $\pazocal{C}^3$ compact boundary $\partial\Omega_{rig}$, see \cref{SchOpen}. Let $\Gamma_D$ be the part of the boundary where a homogenous Dirichlet conditions applies (blue part), $\Gamma_N$ the part where a non-homogenous Neumann condition $\tauu$ applies (orange part), $\Gamma_C$ the potential contact zone (green part), and $\Gamma$ the rest of the boundary, which is free of any constraint (i.e. homogenous Neumann boundary condition). Those four parts are mutually disjoint and moreover: $\overline{\Gamma_D} \cup \overline{\Gamma_N} \cup \overline{\Gamma_C} \cup \overline{\Gamma} = \partial \Omega$. In order to avoid technical difficulties, it is assumed that $\overline{\Gamma_C} \cap \overline{\Gamma_D} = \emptyset$.

The outward normal to $\Omega$ is denoted $\normalInt$. Similarly, the inward normal to $\Omega_{rig}$ is denoted $\normalExt$.

\subsection{Notations and function spaces}\label{sec:esp}
Throughout this article, for any $\pazocal{O}\subset \mathbb{R}^d$, $L^p(\pazocal{O})$ represents the usual set of $p$-th power measurable functions on $\pazocal{O}$, and $\left(L^p(\pazocal{O})\right)^d = \Ll^p(\pazocal{O})$. The scalar product defined on $L^2(\pazocal{O})$ or $\Ll^2(\pazocal{O})$ is denoted (without distinction) by  $\prodL2{\cdot,\cdot}{\pazocal{O}}$ and its norm $\|\cdot\|_{0,\pazocal{O}}$. 

The Sobolev spaces, denoted $W^{m,p}(\pazocal{O})$ with $p\in [1,+\infty]$, $p$ integer are defined as 
$$
W^{m,p}(\pazocal{O}) = \left\{u\in L^p(\pazocal{O})\: : \: D^{\alpha} u \in L^p(\pazocal{O})\ \forall |\alpha|\le m\right\},
$$
where $\alpha$ is a multi-index in $\mathbb{N}^d$ and $\Ww^{m,p}(\pazocal{O})= \left(W^{m,p}(\pazocal{O})\right)^d$. The spaces $W^{s,2}(\pazocal{O})$ and $\Ww^{s,2}(\pazocal{O})$, $s\in \mathbb{R}$, are denoted $H^s(\pazocal{O})$ and $\Hh^s(\pazocal{O})$ respectively. Their norm are denoted $\|\cdot\|_{s,\pazocal{O}}$. 

The subspace of functions in $H^s(\pazocal{O})$ and $\Hh^s(\pazocal{O})$ that vanish on a part of the boundary $\gamma\subset\partial \pazocal{O}$ are denoted $H^s_\gamma(\pazocal{O})$ and $\Hh^s_\gamma(\pazocal{O})$. In particular, we denote the vector space of admissible displacements $\Xx:=\Hh^1_{\Gamma_D}(\Omega)$, and $\Xx^*$ its dual.

In order to fit the notations of functions spaces, vector-valued functions are denoted in bold. For example, $w\in L^2(\Omega)$ while $\ww\in \Ll^2(\Omega)$.

For any $v$ vector in $\mathbb{R}^d$, the product with the normal $v \cdot \normalInt$ (respectively with the normal to the rigid foundation $v \cdot \normalExt$) is denoted $v_{\normalInt}$ (respectively $v_{\normalExt}$). Similarly, the tangential part of  $v$ is denoted $v_{\tanInt} = v - v_{\normalInt} \normalInt$ (respectively $v_{\tanExt} = v - v_{\normalExt} \normalExt$).

Finally, the space of second order tensors in $\mathbb{R}^d$, i.e. the space of linear maps from $\mathbb{R}^d$ to $\mathbb{R}^d$, is denoted $\mathbb{T}^2$. In the same way, the space of fourth order tensors is denoted $\mathbb{T}^4$.
\subsection{Mechanical model}\label{sec:meca}

In this work the material is assumed to verify the linear elasticity hypothesis (small deformations and Hooke's law, see for example \cite{Cia1988}), associated with the small displacements assumption (see \cite{kikuchi1988contact}).
The physical displacement is denoted $\uu$, and belongs to $\Xx$.
The stress tensor is defined by $\sigmaa(\uu) = \Aa : \epsilonn(\uu)$, where $\boldsymbol{\epsilon}(\uu)=\frac{1}{2}(\gradd\uu+\gradd\uu^T)$ denotes the linearized strain tensor, and $\Aa$ is the elasticity tensor. This elasticity tensor is a fourth order tensor belonging to $L^\infty(\Omega,\mathbb{T}^4)$, and it is assumed to be elliptic (with constant $\alpha_0>0$). Regarding external forces, the body force $\ff \in \Ll^2(\Omega)$, and traction (or surface load) $\tauu \in \Ll^2(\Gamma_N)$. 
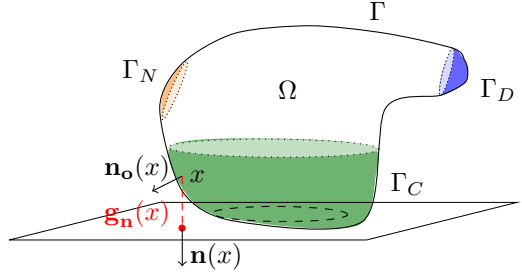
\begin{figure}
\begin{center}
\begin{tikzpicture}

\fill [blue!60] plot [smooth cycle] 
coordinates {(3.9,1.8) (4.055,1.64) (4.095,1.5) (4.055,1.34) (3.73,1.175) (3.76,1.5)};

\draw[black,fill=blue!15, thin,densely dotted,rotate=-15] (3.3,2.425) ellipse (0.05cm and .33cm);

\fill [orange!50] plot [smooth  cycle] 
coordinates {(0.2,1.22) (0.3,1.6) (0.05,1.22) (0.03,0.9) (0.15, 1.)};

\draw[black,fill=orange!15, thin,densely dotted,rotate=-23] (-0.3,1.22) ellipse (0.05cm and .42cm);

\fill [dartmouthgreen!60] plot [smooth  cycle] 
coordinates {(2.795,.50) (2.87,.1) (2.65,-.55) (1.5, -.55) (.5,-.3) (0.15,.5) (1.5, .6)};

\draw[black, fill=dartmouthgreen!25, thin,dotted, rotate=-0.5] (1.51,0.485) ellipse (1.4cm and .12cm);

\draw [black] plot [smooth cycle] 
coordinates {(0,1) (0.3, 1.6) (1,2) (2, 2.1) (3,2) (3.9,1.8) (4.055,1.64) (4.095,1.5) (4.055,1.34) (3.73,1.175) (3,1) (2.7,-.5) (1,-.5) (0.3, -.1)};

\draw[black, fill=dartmouthgreen!60,thin,dashed,rotate=-.5] (1.6,-.39) ellipse (.9cm and .1cm);

\draw[black, ultra thin] (-2,-0.75) -- (2.75,-0.75) -- (4.75,-.25) -- (0, -.25) -- (-2,-0.75);

\node[] at (2.9,2.3) {$\Gamma$};
\node[] at (4.5,1.25) {$\Gamma_D$};
\node[] at (-.25,1.5) {$\Gamma_N$};
\node[] at (3.3,0.) {$\Gamma_C$};
\node[] at (1.7,1.25) {$\Omega$};
\draw[->] (0.3, 0.1) -- (-0.1, -0.1) ;
\node[] at (0.5, 0.1) {$x$};
\node[] at (-0.3, 0.2) {$\normalInt(x)$};
\draw[red, densely dashed] (0.3, 0.1) -- (0.3, -0.6) ;
\node[red] at (0.3,-0.6) {\tiny{$\bullet$}};
\node[red] at (-0.3, -0.4) {$\gG_{\normalExt}(x)$};
\draw[->] (0.3 , -0.6) -- (0.3, -1.1) ;
\node[] at (.75, -0.97) {$\normalExt(x)$};

\end{tikzpicture}
\end{center}
  \caption{Elastic body in contact with a rigid foundation.}
  \label{SchOpen}
\end{figure}

\subsection{Non-penetration condition} At each point $x$ of $\Gamma_C$, let us define the gap $\gG_{\normalExt}(x)$, as the oriented distance function to $\Omega_{rig}$ at $x$, see \cref{SchOpen}. 
Due to the regularity of the rigid foundation, there exists $h$ sufficiently small such that
\begin{equation*}
    \partial\Omega_{rig}^h := \{ x\in\mathbb{R}^d \: : \: |\gG_{\normalExt}(x)| < h \}  \:, 
\end{equation*}
is a neighbourhood of $\partial\Omega_{rig}$ where $\gG_{\normalExt}$ is of class $\pazocal{C}^3$, see \cite{delfour1995boundary}. In particular, this ensures that $\normalExt$ is well defined on $\partial\Omega_{rig}^h$, and that $\normalExt\in \pazocal{C}^2(\partial\Omega_{rig}^h,\mathbb{R}^d)$. Moreover, in the context of small displacements, it can be assumed that the potential contact zone $\Gamma_C$ is such that $\Gamma_C \subset \partial\Omega_{rig}^h$. Hence there exists a neighbourhood of $\Gamma_C$ such that $\gG_{\normalExt}$ and $\normalExt$ are of class $\pazocal{C}^3$ and $\pazocal{C}^2$, respectively.

The non-penetration condition can be stated as follows: $\uu_{\normalExt}\leq \gG_{\normalExt}$ a.e.$\!$ on $\Gamma_C$. Thus, we introduce the closed convex set of admissible displacements that realize this condition, see \cite{eck2005unilateral}:
\begin{displaymath}
	\Kk:=\{\vv \in \Xx \: : \: \vv_{\normalExt}\leq \gG_{\normalExt} \:\:\mbox{a.e.$\!$ on}\: \Gamma_C\} .
\end{displaymath}
\subsection{Mathematical formulation of the problem}
Let us introduce the bilinear and linear forms $a : \Xx \times \Xx \rightarrow \mathbb{R}$ and $L : \Xx \rightarrow \mathbb{R}$, such that:
\begin{equation*}
  a(\uu,\vv) := \int_\Omega \Aa : \epsilonn(\uu) : \epsilonn(\vv) \:, \hspace{1.5em} L(\vv) := \int_\Omega \ff \vv + \int_{\Gamma_N} \tauu \vv \:.
\end{equation*}
According to the assumptions of the previous sections, one is able to show (see \cite{Cia1988}) that $a$ is $\Xx$-elliptic with constant $\alpha_0$ (ellipticity of $\Aa$ and Korn's inequality), symmetric, continuous, and that $L$ is continuous (regularity of $\ff$ and $\tauu$). 

The unknown displacement $\uu$ of the frictionless contact problem is the minimizer of the total mechanical energy of the elastic body, which reads, in the case of pure sliding (unilateral) contact problems:
\begin{equation}
	\underset{\vv\in \Kk}{\inf} \:\: \varphi(\vv) \: := \underset{\vv\in \Kk}{\inf} \:\: \frac{1}{2}a(\vv,\vv)-L(\vv) \:.
 	\label{OPT0}
\end{equation}
It is clear that the space $\Xx$, equipped with the usual $\Hh^1$ norm, is a Hilbert space. Moreover, under the conditions of the previous section, since $\Kk$ is obviously non-empty and the energy functional is strictly convex, continuous and coercive, we are able to conclude (see e.g.$\!$ \cite[Chapter 1]{ekeland1999convex}) that $\uu$ solution of \eqref{OPT0} exists and is unique.

It is well known that \eqref{OPT0} may be rewritten as a variational inequality (of the first kind):
\begin{equation}
	a(\uu,\vv-\uu) \:\geq\: L(\vv-\uu), \:\:\: \forall \vv \in \Kk\: .
    \label{IV0}
\end{equation}
Moreover, as $\ff\in\Ll^2(\Omega)$ and $\tauu\in \Ll^2(\Gamma_N)$, it can be shown (see \cite{DuvLio1972}) that \eqref{OPT0} and \eqref{IV0} are also equivalent to the strong formulation: 
\begin{subequations} \label{FF0:all}
    \begin{align}
    - \Divv \sigmaa(\uu) &= \ff & \mbox{in } \Omega, \label{FF0:1}\\
      \uu                &= 0   & \mbox{on } \Gamma_D,  \label{FF0:2}\\
    \sigmaa(\uu) \cdot \normalInt &= \tauu & \mbox{on } \Gamma_N,   \label{FF0:3}\\
    \sigmaa(\uu) \cdot \normalInt &= 0 & \mbox{on } \Gamma,   \label{FF0:4}\\
     \uu_{\normalExt} \leq \gG_{\normalExt}, 
     \sigmaa_{\normalInt\!\normalExt}(\uu) &\leq 0, \sigmaa_{\normalInt\!\normalExt}(\uu)(\uu_{\normalExt}-\gG_{\normalExt})=0 & \mbox{on } \Gamma_C,   \label{FF0:5}\\
    \sigmaa_{\normalInt\!\tanExt}(\uu) &= 0 & \mbox{on } \Gamma_C,  \label{FF0:6}
    \end{align}
\end{subequations}
where $\sigmaa_{\normalInt\!\normalExt}(\uu) = \sigmaa(\uu)\cdot\normalInt\cdot\normalExt$ and $\sigmaa_{\normalInt\!\tanExt}(\uu) = \sigmaa(\uu)\cdot\normalInt-\sigmaa_{\normalInt\!\normalExt}(\uu)\normalExt$ are the normal and tangential constraints on $\Gamma_C$.
\begin{remark}
	Note that existence and uniqueness of the solution to \eqref{OPT0} $\uu\in \Xx$ also holds under weaker assumptions on the data, namely $\ff\in\Xx^*$ and $\tauu\in \Hh^{-\frac{1}{2}}(\Gamma_N)$ (under the appropriate modifications in the definition of $L$). Here, we choose the minimal regularity that ensures equivalence between \eqref{IV0} and \eqref{FF0:all}. Regarding regularity results, the reader is referred to \cite{kinderlehrer1981remarks}.
\end{remark}

\begin{remark}
  Conditions \eqref{FF0:5} and \eqref{FF0:6} may seem different from the usual
  \begin{subequations} \label{CL0:all}
    \begin{align}
       \uu_{\normalInt} \leq \gG_{\normalInt}, \hspace{1em}\sigmaa_{\normalInt\! \normalInt}(\uu) &\leq 0, \hspace{1em}\sigmaa_{\normalInt\! \normalInt}(\uu)(\uu_{\normalInt}-\gG_{\normalInt})=0 & \mbox{on } \Gamma_C, \label{CL0:1}\\
      \sigmaa_{\normalInt\! \tanInt}(\uu) &= 0 & \mbox{on } \Gamma_C,  \label{CL0:2}
    \end{align}
  \end{subequations}
  where $\gG_{\normalInt}(x)$ denotes the distance between $x\in \Gamma_C$ and the rigid foundation computed in the direction of the normal $\normalInt$ to $\Gamma_C$. 
  %
  %
  Actually, since we assume the deformable body undergoes small displacements relative to its reference configuration, both sets of conditions are equivalent. 
  %
  %
  %
  More specifically, from the small displacement hypothesis, the normal vector $\normalExt$ and the gap $\gG_{\normalExt}$ to the rigid foundation can be replaced by $\normalInt$ and $\gG_{\normalInt}$ 
  (we refer to \cite[Chapter 2]{kikuchi1988contact} for the details).

  Therefore, in our context, writing the formulation associated to the contact problem using $\normalInt$ or $\normalExt$ makes absolutely no difference. We choose the latter formulation because it proves itself very convenient when dealing with shape optimization, see \cref{sec:shapeopt}.
  \label{RemHypHPP}
\end{remark}

\subsection{Friction condition}
Let $\mathfrak{F} : \Gamma_C \rightarrow \mathbb{R}$, $\mathfrak{F} > 0$, be the friction coefficient. The basis of Tresca model is to replace the usual Coulomb threshold $|\sigmaa_{\normalInt \!\normalExt}(\uu)|$ by a fixed strictly positive function $s$, which leads to the following conditions on $\Gamma_C$:
\begin{equation}
  \left\{
	\hspace{0.5em}    
    \begin{aligned}
    	|\sigmaa_{\normalInt\!\tanExt}(\uu)| \:&<\: \mathfrak{F} s & \:\: \mbox{on } \{ x \in \Gamma_C \: : \: \uu_{\tanExt}(x) = 0 \}\:, \\
    	\sigmaa_{\normalInt\!\tanExt}(\uu) \:&=\: -\mathfrak{F} s  \frac{\uu_{\tanExt}}{|\uu_{\tanExt}|} & \:\: \mbox{on } \{ x \in \Gamma_C \: : \: \uu_{\tanExt}(x) \neq 0 \}\:,
    	\label{CL1}
    \end{aligned}
  \right.    	
\end{equation}
which represent respectively \textit{sticking} and \textit{sliding} points.
\begin{remark}
	Of course, replacing the Coulomb threshold by the fixed function $s$ leads to a simplified and approximate model of friction. Especially, in the Tresca model, there may exist points $x\in \Gamma_C$ such that $\sigmaa_{\normalInt\!\normalExt}(\uu)(x) = 0$ and $\uu_{\tanExt}(x)\neq 0$, in which case $\sigmaa_{\normalInt\!\tanExt}(\uu)(x)\neq 0$. In other words, friction can occur even if there is no contact.
\end{remark}
In order to avoid regularity issues, it is assumed that $\mathfrak{F}$ is uniformly Lipschitz continuous and $s \in L^2(\Gamma_C)$. Before stating the minimization problem in this case, let us introduce the non-linear functional $j_T : \Xx \rightarrow \mathbb{R} $ defined by:
\begin{equation*}
	j_T(\vv) := \int_{\Gamma_C} \mathfrak{F} s |\vv_{\tanExt}| \: .
\end{equation*}
With these notations, since considering the Tresca friction model means taking into account the frictional term $j_T$ in the energy functional, the associated minimization problem writes:
\begin{equation}
	\underset{\vv\in \Kk}{\inf} \:\: \varphi(\vv)+j_T(\vv) \:.
 	\label{OPT1}
\end{equation}
Since the additional term $j_T$ in the functional is convex, positive and  continuous 
one can deduce existence and uniqueness of the solution $\uu\in \Xx$, see for example \cite[Section 1.5]{oden1980theory}.
From this reference, one also gets that \eqref{OPT1} can be equivalently rewritten as a variational inequality (of the second kind):
\begin{equation}
	a(\uu,\vv-\uu) + j_T(\vv) - j_T(\uu) \:\geq\: L(\vv-\uu)\:, \:\:\: \forall \vv \in \Kk\: .
    \label{IV1}
\end{equation}
Again, from \cite{DuvLio1972} problems \eqref{OPT1} and \eqref{IV1} are equivalent to the strong formulation \eqref{FF0:all}, except for the last condition \eqref{FF0:6}, which is replaced by the two conditions \eqref{CL1}.

\begin{remark}
	Since $\uu\in \Xx$, the regularity of $\sigmaa(\uu)\cdot\normalInt$ is in general $\Hh^{-\frac{1}{2}}(\Gamma_C)$. However, even though no better regularity can be expected for the normal component $\sigmaa_{\normalInt\!\normalExt}$ in the general case, one has that the tangential component $\sigmaa_{\normalInt\!\tanExt}\in \Ll^2(\Gamma_C)$ when $s\in L^2(\Gamma_C)$. We refer to \cite[Chapter 4]{stadler2004infinite} for further details.
\end{remark}

\subsection{Penalty formulation}

The formulation that will be studied here originate from the classical penalty method, see \cite{Lio1969} and \cite{aubin2007approximation} for the general method,  \cite{kikuchi1981penalty} or \cite{kikuchi1988contact} for its application to unilateral contact problems, and \cite{chouly2013convergence} for its application to the Tresca friction problem. 
This formulation reads: find $\uu_\varepsilon$ in $\Xx$ such that, for all $\vv \in \Xx$,
\begin{equation}
  a(\uu_\varepsilon,\vv) + \frac{1}{\varepsilon} \prodL2{\maxx(\uu_{\varepsilon,\normalExt} -\gG_{\normalExt}), \vv_{\normalExt}}{\Gamma_C}  + \frac{1}{\varepsilon} \prodL2{\qq(\varepsilon\mathfrak{F}s,\uu_{\varepsilon,\tanExt}), \vv_{\tanExt}}{\Gamma_C} = L(\vv) \:,
  \label{FV}
\end{equation}
where $\maxx$ denotes the projection onto $\mathbb{R}_+$ in $\mathbb{R}$ (also called the positive part function) and, for any $\alpha\in\mathbb{R}_+$, $\qq(\alpha,\cdot)$ denotes the projection onto the ball $\mathcal{B}(0,\alpha)$ in $\mathbb{R}^{d-1}$. Those projections admit analytical expressions: for all $y\in\mathbb{R}$, $z\in\mathbb{R}^{d-1}$:
\begin{equation*}
    \maxx(y):=\max\{0,y\} \:,
\qquad
    \qq(\alpha,z) := \left\{
    \begin{array}{lr}
         z & \mbox{ if } |z|\leq \alpha ,\\
         \alpha\dfrac{z}{|z|} & \mbox{ else.}
    \end{array}
    \right.
\end{equation*}
It is well known (see for example \cite{kikuchi1981penalty,chouly2013convergence} or \cite[Section 6.5]{kikuchi1988contact}) that \eqref{FV} admits a unique solution $\uu_\varepsilon\in\Xx$. Moreover, from the same references, one gets that passing to the limit $\varepsilon \to 0$ leads to $\uu_\varepsilon \to \uu$ strongly in $\Xx$.

\begin{remark}
    Formulation \eqref{FV} is actually the optimality condition related to the unconstrained differentiable optimization problem derived from \eqref{OPT1}:
$$
   \underset{\vv \in \Xx}{\inf} \left\{ \varphi(\vv) + j_{T,\varepsilon}(\vv) + j_\varepsilon(\vv) \right\} \:,
$$
where $j_\varepsilon$ is a penalty term introduced to relax the constraint $\vv\in\Kk$, and $j_{T,\varepsilon}$ is a regularization of $j_T$.

Moreover, in this model, one gets from \eqref{FV} that the non-penetration conditions \eqref{FF0:5}-\eqref{FF0:6} and the friction condition \eqref{CL1} rewrite:
  \begin{subequations} \label{CL2:all}
    \begin{align}
       \sigmaa_{\normalInt\! \normalExt}(\uu_\varepsilon) &= -\frac{1}{\varepsilon} \maxx(\uu_{\varepsilon,\normalExt} -\gG_{\normalExt}) & \mbox{on } \Gamma_C, \label{CL2:1}\\
      \sigmaa_{\normalInt\! \tanExt}(\uu_\varepsilon) &= - \frac{1}{\varepsilon} \qq(\varepsilon\mathfrak{F}s,\uu_{\varepsilon,\tanExt}) & \mbox{on } \Gamma_C.  \label{CL2:2}
    \end{align}
  \end{subequations}
  From those expressions, one deduces the new definitions for the sets of points of particular interest:
  \begin{itemize}
  	\item points in contact: $\{ x \in \Gamma_C \: |\: \uu_{\varepsilon,\normalExt}\geq \gG_{\normalExt} \}$,
  	\item sticking points: $\{ x \in \Gamma_C \: |\:\: |\uu_{\varepsilon,\tanExt}|\leq \varepsilon\mathfrak{F}s \}$,
  	\item sliding points: $\{ x \in \Gamma_C \: |\:\: |\uu_{\varepsilon,\tanExt}|\geq \varepsilon\mathfrak{F}s \}$.
  \end{itemize}
\end{remark}

\section{Shape optimization}
\label{sec:shapeopt}

Given a cost functional $J(\Omega)$ depending explicitly on the domain $\Omega$, and also implicitly, through $y(\Omega)$ the solution of some variational problem on $\Omega$, the optimization of $J$ with respect to $\Omega$ or \textit{shape optimization problem} reads:
\begin{equation}
    \inf_{\Omega \in \pazocal{U}_{ad}} J(\Omega) \:,
    \label{ShapeOPT}
\end{equation}
where $\pazocal{U}_{ad}$ stands for the set of admissible domains. 

Here, since the physical problem considered is modeled by \eqref{FV}, one has $y(\Omega)=\uu_\varepsilon(\Omega)$ solution of \eqref{FV} defined on $\Omega$. Therefore, let us replace the notation of the functional $J$ by $J_\varepsilon$ to emphasize the dependence with respect to the penalty parameter.  
%
%
Let $D\subset \mathbb{R}^d$ be a fixed bounded smooth domain, and let $\hat{\Gamma}_D\subset\partial D$ be a part of its boundary which will be the "potential" Dirichlet boundary. This means that for any domain $\Omega\subset D$, the Dirichlet boundary associated to $\Omega$ will be defined as $\Gamma_D:=\partial\Omega \cap \hat{\Gamma}_D$. With these notations, we introduce the set $\pazocal{U}_{ad}$ of all admissible domains, which consists of all smooth open domains $\Omega$ such that the Dirichlet boundary $\Gamma_D\subset \partial D$ is of stritly positive measure, that is:
$$
   \pazocal{U}_{ad} := \{ \Omega \subset D \: : \: \Omega \mbox{ is of class $\pazocal{C}^1$ and } |\partial\Omega \cap \hat{\Gamma}_D| > 0 \}.
$$
\subsection{Derivatives}
\label{sec:deriv}
The shape optimization method followed in this work is the so-called \textit{perturbation of the identity}, as presented in~\cite{murat1975etude} and \cite{henrot2006variation}. Let us introduce $\Cc^1_b(\mathbb{R}^d) := {(\pazocal{C}^1(\mathbb{R}^d)\cap W^{1,\infty}(\mathbb{R}^d) )}^d$, equipped with the $d$-dimensional $W^{1,\infty}$ norm, denoted $\norml\cdot \normr_{1,\infty}$. In order to move the domain $\Omega$, let $\thetaa \in \Cc^1_b(\mathbb{R}^d)$ be a (small) geometric deformation vector field. The associated perturbed or transported domain in the direction $\thetaa$ will be defined as: $\Omega(t) := (\Id+t\thetaa)(\Omega)$ for any $t>0$. To make things clear
some basic notions of shape sensitivity analysis from \cite{sokolowski1992introduction} are briefly recalled. 

We denote again $y(\Omega)$ the solution, in some Sobolev space denoted $W(\Omega)$, of a variational formulation posed on $\Omega$. For any fixed $\thetaa$, for any small $t>0$, let $y(\Omega(t))$ be the solution of the same variational formulation posed on $\Omega(t)$. 
If the variational formulation is regular enough (e.g.$\!$ if it is linear), it can be proved (see~\cite[Chapter 3]{sokolowski1992introduction}) that $y(\Omega(t))\circl(\Id+t\thetaa)$ also belongs to $W(\Omega)$.
\begin{itemize}[leftmargin=*]
    \item The \textit{Lagrangian derivative} or \textit{material derivative} of $y(\Omega)$ in the direction $\thetaa$ is the element $\dot{y}(\Omega)[\thetaa] \in W(\Omega)$ defined by:
    $$
        \dot{y}(\Omega)[\thetaa] := \lim_{t\searrow 0} \:\frac{1}{t}\left( y(\Omega(t))\circl(\Id+t\thetaa) - y(\Omega)\right) \: .
        \label{DefMatDer}
    $$
    If the limit is computed weakly in $W(\Omega)$ (respectively strongly), we talk about \textit{weak} material derivative (respectively \textit{strong} material derivative).
    \item If the additional condition $\gradd y(\Omega)\thetaa \in W(\Omega)$ holds for all $\thetaa \in \Cc^1_b(\mathbb{R}^d)$, then one may define a directional derivative called the \textit{Eulerian derivative} or \textit{shape derivative} of $y(\Omega)$ in the direction $\thetaa$ as the element $dy(\Omega)[\thetaa]$ of $W(\Omega)$ such that:
    $$
        dy(\Omega)[\thetaa] := \dot{y}(\Omega)[\thetaa] - \gradd y(\Omega)\thetaa \: .
        \label{DefShaDer}
    $$
    \item The solution $y(\Omega)$ is said to be \textit{shape differentiable} if it admits a directional derivative for any admissible direction $\thetaa$, and if the map $\thetaa \mapsto dy(\Omega)[\thetaa]$ is linear continuous from $\Cc^1_b(\mathbb{R}^d)$ to $W(\Omega)$.
\end{itemize}
%
\begin{remark}
    Linearity and continuity of $\thetaa \mapsto \dot{y}(\Omega)[\thetaa]$ is actually equivalent to the Gâteaux differentiability of the map $\thetaa \mapsto y(\Omega(\thetaa))\circl(\Id+\thetaa)$. The reader is referred to \cite[Chapter 8]{DelZol2001} for a complete review on the different notions of derivatives.
\end{remark}
When there is no ambiguity, the material and shape derivatives of some function $y$ at $\Omega$ in the direction $\thetaa$ will be simply denoted $\dot{y}$ and $dy$, respectively. 

\subsection{Shape sensitivity analysis of the penalty formulation}

The goal of this section is to prove the differentiability of $\uu_\varepsilon$ with respect to the shape. 
As functions $\maxx$ and $\qq$ fail to be Fréchet differentiable it is not possible to rely on the implicit function theorem as in \cite[Chapter 5]{henrot2006variation}.
%
Nevertheless, these functions admit directional derivatives. %
Hence, working with the directional derivatives of $\maxx$ and $\qq$ and following the approach in \cite{sokolowski1992introduction}, we show existence of directional material/shape derivatives for $\uu_\varepsilon$. 
Then, under assumptions on some specific subsets of $\Gamma_C$ (this will be presented and referred to as \cref{A1}), shape differentiability of $\uu_\varepsilon$ is proved.

Since the domain is transported, the functions $\Aa$, $\ff$, $\tauu$, $\mathfrak{F}$, and $s$ have to be defined everywhere in $\mathbb{R}^d$. 
They also need to enjoy more regularity for usual differentiability results to hold. In particular 
we make the following regularity assumptions :
\begin{hypothesis}\label{A0}
    $\Aa\in \pazocal{C}^1_b(\mathbb{R}^d,\mathbb{T}^4),\:\, \ff \in \Hh^1(\mathbb{R}^d),\:\, \tauu \in \Hh^2(\mathbb{R}^d)$, $s\in L^2(\Gamma_C)$ and $\mathfrak{F}s \in H^2(\mathbb{R}^d)$.
\end{hypothesis}
%
\begin{notation} Through change of variables, we can transform expression on $\Omega(t)$ to expression on 
$\Omega$. 
Composition with the operator $\circl (\Id+t\thetaa)$ will be denoted by $(t)$, for instance, $\Aa(t):=\Aa\circl(\Id+t\thetaa)$. 
The normal and tangential component associated to $\normalExt(t)$ of a vector $v$ is denoted $v_{\normalExt(t)}$ and $v_{\tanExt(t)}$ respectively. 
For integral expressions the Jacobian and tangential Jacobian of the transformation gives $\JacV(t) := $ Jac$(\Id+t\thetaa)$ and $\JacB(t):=$ Jac$_{\Gamma(t)}(\Id+t\thetaa)$. 
To simplify the notations let $\uu_{\varepsilon,t} := \uu_\varepsilon(\Omega(t))$ and 
$\uu_{\varepsilon}^t := \uu_{\varepsilon,t}\circl(\Id+t\thetaa)$. Finally, we also introduce the map $\Phi_\varepsilon:\mathbb{R}_+ \to \Xx$ such that for each $t>0$, $\Phi_\varepsilon(t)=\uu_{\varepsilon}^t$.
\end{notation}
%

As differentiability of $\uu_\varepsilon$ with respect to the shape is directly linked to differentiability of $\Phi_\varepsilon$, we will focus on the latter. However, note that the direction $\thetaa$ is fixed in the definition of $\Phi_\varepsilon$, therefore every property of $\Phi_\varepsilon$ (continuity, differentiability) will be associated to a directional property for $\uu_\varepsilon$. 

\subsubsection{Continuity of $\Phi_\varepsilon$}

Before getting interested in differentiability, the first step is to prove continuity.

\begin{theo}
  If \cref{A0} holds, then for any $\thetaa \in \Cc^1_b(\mathbb{R}^d)$, $\Phi_\varepsilon$ is strongly continuous at $t=0^+$.
\end{theo}

\begin{proof} When $t$ is small enough, the transported potential contact zone verifies $\Gamma_C(t)\subset\partial\Omega_{rig}^h$, so that the regularities of $\gG_{\normalExt}$ and $\normalExt$ are preserved. 
When transported to $\Omega(t)$, problem \eqref{FV} becomes: find $\uu_{\varepsilon,t} \in \Hh^1_{\Gamma_D(t)}(\Omega(t))=:\Xx(t)$ such that, 
\begin{equation} 
    \begin{aligned}
        \int_{\Omega(t)}& \Aa : \epsilonn(\uu_{\varepsilon,t}) : \epsilonn(\vv_t) + \frac{1}{\varepsilon} \int_{\Gamma_C(t)} \maxx(\uu_{\varepsilon,t} \cdot \normalExt -\gG_{\normalExt}) (\vv_t )_{\normalExt} \\
        \: & + \frac{1}{\varepsilon} \int_{\Gamma_C(t)}  \qq\left(\varepsilon\mathfrak{F}s,(\uu_{\varepsilon,t})_{\tanExt}\right)(\vv_t)_{\tanExt} = \int_{\Omega(t)} \ff \: \vv_t + \int_{\Gamma_N(t)} \tauu \: \vv_t \quad \forall \vv_t \in \Xx(t).
  \end{aligned}
  \label{FVT0}
\end{equation}
We can transform \eqref{FVT0} as an expression on the reference domain $\Omega$. 
For the test function we use $\vv^t:=\vv_t\circl(\Id+t\thetaa)$. Moreover, to simplify the expressions, we introduce
\begin{align*}
    R_{\normalExt}(\vv) &:=\maxx(\vv_{\normalExt}-\gG_{\normalExt})\:, \quad R_{\normalExt}^{t}(\vv) :=\maxx(\vv_{\normalExt(t)}-\gG_{\normalExt}(t))\:,\\
    S_{\tanExt}(\vv) &:= \qq(\varepsilon\mathfrak{F}s,\vv_{\tanExt})\:, \quad S_{\tanExt}^t(\vv) := \qq(\varepsilon(\mathfrak{F}s)(t),\vv_{\tanExt(t)})\:,
\end{align*}
and finally, the transported strain tensor $\epsilonn^t$ is also introduced: for all $\vv\in\Xx$,
\begin{equation*}
    \epsilonn^t(\vv):= \frac{1}{2} \left( \gradd \vv{(\Ii+t\gradd\thetaa)}^{-1} + {(\Ii+t\gradd\thetaa^T)}^{-1}{\gradd \vv}^T  \right) \: .
\end{equation*}
With the notations introduced and the change of variables mentionned above we have
\begin{equation}
    \begin{aligned}
        \int_{\Omega} \Aa(t) : &\epsilonn^t(\uu_{\varepsilon}^{t}) : \epsilonn^t(\vv^t) \:\JacV(t) + \frac{1}{\varepsilon} \int_{\Gamma_{C}} R_{\normalExt}^{t}(\uu_{\varepsilon}^{t}) \vv^t_{\normalExt(t)}  \:\JacB(t) 
        \\ & 
        + \frac{1}{\varepsilon} \int_{\Gamma_{C}} S_{\tanExt}^t(\uu_{\varepsilon}^{t}) \vv^t_{\tanExt(t)} \JacB(t) 
        =  \int_{\Omega} \ff(t)\vv^t \JacV(t) + \int_{\Gamma_{N}} \tauu(t)\vv^t \JacB(t) \:.
    \end{aligned}
    \label{FVT}
\end{equation}
Note that for $t$ sufficiently small, $\norml t \thetaa \normr_{1,\infty} \! < 1$. Thus the application $(\Id+t\thetaa)$ is a $\pazocal{C}^1$-diffeomorphism, and so the map $\vv_t \mapsto \vv^t$ is an isomorphism from $\Xx(t)$ to $\Xx$. Thus, one deduces that $\uu_{\varepsilon}^{t}$ is the solution of the variational formulation obtained when replacing $\vv^t$ by $\vv$ in \eqref{FVT}, which holds for all $\vv$ in $\Xx$.

\paragraph{Uniform boundedness of $\uu_\varepsilon^t$ in $\Xx$.} 
Let us show that $\uu_\varepsilon^t$ is uniformly bounded in $t$. To achieve this we use the first order Taylor expansions with respect to $t$ of all known terms in \eqref{FVT}. Such expansions are valid due to \cref{A0} and the regularity assumptions on $\Omega$, see \cite{henrot2006variation,sokolowski1992introduction}. We recall some of them: $\forall \vv \in \Xx$,
\begin{equation*}
    \begin{aligned}
        \left\lVert \: \epsilonn^t(\vv) - \epsilonn(\vv) + \frac{t}{2}\left( \gradd \vv \gradd \thetaa + {\gradd \thetaa}^T {\gradd \vv}^T \right)\right\lVert_{0,\Omega} &= O(t^2)\norml \vv \normr_{\Xx} \:,\\
        \norml \Aa(t) - \Aa - t\gradd \Aa :\thetaa \normr_{\infty,\Omega} &= O(t^2)\:, \\
        \norml \JacV(t)-1-t\divv\thetaa \normr_{\infty,\Omega} &= O(t^2)\:, \\
        \norml \JacB(t)-1-t\divv_{\Gamma}\thetaa \normr_{\infty,\partial\Omega} &= O(t^2)\:, \\
        \norml \vv_{\normalExt(t)}- \vv_{\normalExt} - t(\vv\cdot(\gradd \normalExt \thetaa))\normr_{0,\Gamma_C} &= O(t^2)\norml \vv \normr_{0,\Gamma_C}\:,\\
        \norml \vv_{\tanExt(t)}- \vv_{\tanExt} + t(\vv\cdot(\gradd \normalExt \thetaa))\normalExt + t(\vv\cdot\normalExt)(\gradd \normalExt \thetaa) \normr_{0,\Gamma_C} &= O(t^2)\norml \vv \normr_{0,\Gamma_C}\:.
    \end{aligned}
\end{equation*}

Making use of these expansions, the ellipticity of $a$ and taking $\uu_\varepsilon^t$ as test-function in \eqref{FVT}, one gets the following estimate:
\begin{equation*}
    (\alpha_0 + O(t))\norml \uu_\varepsilon^t \normr_{\Xx}^2 \:\leq \: O(t)\norml \uu_\varepsilon^t \normr_{\Xx} + \:O(t^2) \:.
\end{equation*}
Thus, for $t$ small enough, one gets that there exist some positive constants $C_1$ and $C_2$ such that the sequence $C_1\norml \uu_\varepsilon^t \normr_{\Xx}^2-C_2\norml \uu_\varepsilon^t \normr_{\Xx}$ is uniformly bounded in $t$, which proves uniform boundedness of $\{\uu_\varepsilon^{t_k}\}_k$ in $\Xx$, for any sequence $\{t_k\}_k$ decreasing to 0.

\paragraph{Continuity.}
First, one needs to show that the limit (in some sense that will be specified) of $\uu_\varepsilon^t$ as $t\to 0$ is indeed $\uu_\varepsilon$. Let $\{t_k\}_k$ be a sequence decreasing to 0. Since the sequence $\{\uu_\varepsilon^{t_k}\}_k$ is bounded 
and $\Xx$ a reflexive Banach space, there exists a weakly convergent subsequence (still denoted $\{\uu_\varepsilon^{t_k}\}_k$), say $\uu_\varepsilon^{t_k} \rightharpoonup \hat{\uu}_\varepsilon \in \Xx$. 

Due to the Taylor expansions above, the weak convergence of $\{\uu_\varepsilon^{t_k}\}_k$, the compact embedding $\Hh^{\frac{1}{2}}(\Gamma_C) \hookrightarrow \Ll^2(\Gamma_C)$ and Lipschitz continuity of $\maxx$ and $\qq$, taking $t=t_k$ in \eqref{FVT} and passing to the limit $k\to +\infty$ leads to: for all $\vv \in \Xx$,
\begin{equation*}
        \int_{\Omega} \Aa : \epsilonn(\hat{\uu}_{\varepsilon}) : \epsilonn(\vv) + \frac{1}{\varepsilon} \int_{\Gamma_{C}} R_{\normalExt}(\hat{\uu}_{\varepsilon}) \vv_{\normalExt}
        + \frac{1}{\varepsilon} \int_{\Gamma_{C}} S_{\tanExt}(\hat{\uu}_{\varepsilon}) \vv_{\tanExt} =  \int_{\Omega} \ff \: \vv + \int_{\Gamma_{N}} \tauu\: \vv \:.
\end{equation*}
This precisely means that $\hat{\uu}_\varepsilon = \uu_\varepsilon$, since they are both solution of problem \eqref{FV}, which admits a unique solution. The uniqueness also proves that the whole sequence $\{\uu_\varepsilon^{t_k}\}_k$ tends to $\uu_\varepsilon$.

Now, strong continuity of the map $t \mapsto \uu_\varepsilon^t$ at $t=0^+$ in $\Xx$ may be proved using the difference $\boldsymbol{\delta}_{\uu,\varepsilon}^t:=\uu_\varepsilon^t-\uu_\varepsilon$, which appears when subtracting the formulations verified by $\uu_\varepsilon^t$ and $\uu_\varepsilon$, respectively. Note that $\boldsymbol{\delta}_{\uu,\varepsilon}^t$ is bounded in $\Xx$ and that it converges weakly to 0 in $\Xx$. 

For $t$ sufficiently small, let us consider
\begin{equation}
  \begin{aligned}
    \int_{\Omega} &\Aa(t) : \epsilonn^t(\uu_{\varepsilon}^{t}) : \epsilonn^t(\vv) \:\JacV(t) - \int_{\Omega} \Aa : \epsilonn(\uu_\varepsilon) : \epsilonn(\vv) \\
    & + \frac{1}{\varepsilon} \int_{\Gamma_{C}} R_{\normalExt}^{t}(\uu_{\varepsilon}^{t}) \vv_{\normalExt(t)} \: \JacB(t) - \frac{1}{\varepsilon} \int_{\Gamma_{C}} R_{\normalExt}(\uu_{\varepsilon})\vv_{\normalExt} \\
    & + \frac{1}{\varepsilon} \int_{\Gamma_{C}} S_{\tanExt}^t(\uu_{\varepsilon}^{t}) \vv_{\tanExt(t)} \JacB(t) - \frac{1}{\varepsilon} \int_{\Gamma_{C}} S_{\tanExt}(\uu_{\varepsilon}) \vv_{\tanExt} \\
    & = \int_{\Omega} \ff(t) \vv \: \JacV(t) - \int_{\Omega} \ff \: \vv + \int_{\Gamma_{N}} \tauu(t)\vv \: \JacB(t) - \int_{\Gamma_{N}} \tauu \: \vv \:.
  \end{aligned}
  \label{FVTMoinsFV}
\end{equation}
Let us introduce three groups of terms, for any $\vv \in \Xx$, say $T_1(\vv)$, $T_2(\vv)$, $T_3(\vv)$ and $T_4(\vv)$, each $T_i(\vv)$ corresponding to the $i$-th line in equation \eqref{FVTMoinsFV}. The terms $T_1$ and $T_4$ have already been treated in the literature as they appear in the classical elasticity problem. Especially, one gets from \cite[Section 3.5]{sokolowski1992introduction} that
\begin{equation*}
    \begin{aligned}
    T_1(\boldsymbol{\delta}_{\uu,\varepsilon}^t) & \: \geq \: \alpha_0 \norml \boldsymbol{\delta}_{\uu,\varepsilon}^t \normr_{\Xx}^2 - tC\norml \boldsymbol{\delta}_{\uu,\varepsilon}^t \normr_{\Xx} \:, \\
    | T_4(\boldsymbol{\delta}_{\uu,\varepsilon}^t) | & \: \leq \: tC \norml \boldsymbol{\delta}_{\uu,\varepsilon}^t \normr_{\Xx} = o(t) \:.
    \end{aligned}
\end{equation*}
As for $T_2$ and $T_3$, 
some boundedness results are needed which can be deduced from the properties of functions $\maxx$ and $\qq$, the trace theorem and continuity of $t\mapsto \normalExt(t)$, $t\mapsto \gG_{\normalExt}(t)$, $t\mapsto (\mathfrak{F}s)(t)$, $t\mapsto \vv_{\tanExt(t)}$. For any $\vv\in\Xx$, one has
\begin{equation*}
    \begin{aligned}
        \norml \vv_{\normalExt(t)}\normr_{0,\Gamma_C} \leq C \norml \vv \normr_{\Xx}\:, \hspace{2em} &\norml \vv_{\tanExt(t)}\normr_{0,\Gamma_C} \leq C \norml \vv \normr_{\Xx}\:, \\
        \norml R_{\normalExt}^{t}(\vv) \normr_{0,\Gamma_C} \leq C\left(1+\norml \vv \normr_{\Xx}\right)\:, \hspace{2em} &  \norml S_{\tanExt}^{t}(\vv) \normr_{0,\Gamma_C} \leq C \:,
    \end{aligned}
\end{equation*}
and the same inequalities applies to $\vv_{\normalExt}$, $\vv_{\tanExt}$, $R_{\normalExt}$ and $S_{\tanExt}$. 
Then, for  $T_2$
\begin{equation*}
    \begin{aligned}
        T_2(\vv) = & \: \frac{1}{\varepsilon} \int_{\Gamma_{C}} R_{\normalExt}^{t}(\uu_{\varepsilon}^{t})\vv_{\normalExt(t)} \: (\JacB(t)-1) + \frac{1}{\varepsilon} \int_{\Gamma_{C}} R_{\normalExt}^{t}(\uu_{\varepsilon}^{t}) (\vv_{\normalExt(t)}-\vv_{\normalExt}) \\
        & + \frac{1}{\varepsilon} \int_{\Gamma_{C}} \left( R_{\normalExt}^{t}(\uu_{\varepsilon}^{t}) - R_{\normalExt}(\uu_{\varepsilon})\right)\vv_{\normalExt} \:.
    \end{aligned}
\end{equation*}
since $ \left( R_{\normalExt}^{t}(\uu_{\varepsilon}^{t}) - R_{\normalExt}(\uu_{\varepsilon})\right)$ is bounded in $L^2(\Gamma_C)$ and $\boldsymbol{\delta}_{\uu,\varepsilon}^t\to 0$ strongly in $\Ll^2(\Gamma_C)$, 
$$
    \left| T_2(\boldsymbol{\delta}_{\uu,\varepsilon}^t) \right| \leq t\,C \norml \boldsymbol{\delta}_{\uu,\varepsilon}^t \normr_{\Xx} +  \frac{1}{\varepsilon} \int_{\Gamma_{C}} \left|R_{\normalExt}^{t}(\uu_{\varepsilon}^{t}) - R_{\normalExt}(\uu_{\varepsilon})\right| \left|\boldsymbol{\delta}_{\uu,\varepsilon}^t \right|  = o(1) \:.
$$
As for $T_3$, using the boundedness of $\left(S_{\tanExt}^t(\uu_{\varepsilon}^{t}) - S_{\tanExt}(\uu_{\varepsilon}) \right)$ in $\Ll^2(\Gamma_C)$ and the same decomposition as for $T_2$ we get
\begin{equation*}
    | T_3(\boldsymbol{\delta}_{\uu,\varepsilon}^t) | \leq t\,C 
    \norml \boldsymbol{\delta}_{\uu,\varepsilon}^t \normr_{\Xx} 
    +  \frac{1}{\varepsilon} \int_{\Gamma_{C}} \left| S_{\tanExt}^t(\uu_{\varepsilon}^{t}) - S_{\tanExt}(\uu_{\varepsilon}) \right| \left|\boldsymbol{\delta}_{\uu,\varepsilon}^t \right|  = o(1) \:,
\end{equation*}
Thus, choosing $\boldsymbol{\delta}_{\uu,\varepsilon}^t$ as a test-function in \eqref{FVTMoinsFV} yields: $\alpha_0 \norml \boldsymbol{\delta}_{\uu,\varepsilon}^t \normr_{\Xx}^2 \leq o(1)$, which proves strong continuity of $t \mapsto \uu_\varepsilon^t$ in $\Xx$ at $t=0^+$.
\end{proof}

This result means that $\uu_\varepsilon$ is strongly directionally continuous with respect to the shape. Now, it remains to prove that differentiability also holds.


\subsubsection{Directional differentiability of $\maxx$ and $\qq$}
\label{app:DirDiffQ}


In order to study differentiability of $\Phi_\varepsilon$, we need some preliminary results concerning the  directional differentiability of the non-Fréchet differentiable functions $\maxx$ and $\qq$.
Let us briefly recall the definition of a Nemytskij operator. 
\begin{defn}
  Let $S$ be a measurable subset of $\mathbb{R}^d$, let $X$ and $Y$ be two real Banach spaces of functions defined on $S$. Given a mapping $\psi : S\times X \to Y$, the associated Nemytskij operator $\Psi$ is defined by:
  \begin{displaymath}
    \Psi(v)(x) := \psi(x,v(x)) \:, \:\: \mbox{ for all } x \in S \: . 
  \end{displaymath}
\end{defn}
As explained in details in \cite{goldberg1992nemytskij} or \cite[Section 4.3]{troltzsch2010optimal}, the smoothness of $\psi$ does not guarantee the smoothness of $\Psi$. In our case, we are only interested in directional differentiability of the Nemytskij operators associated to $\maxx$ and $\qq$. Thus, directional differentiability and Lipschitz continuity of $\maxx$ and $\qq$ in $\mathbb{R}$ and $\mathbb{R}^d$, combined with Lebesgue's dominated convergence, will enable us to conclude directly, without using the more general results from \cite{goldberg1992nemytskij}.


\begin{lemma}
 The function $\maxx:\mathbb{R}\to\mathbb{R}$ is Lipschitz continuous and directionally differentiable, with directional derivative at $u$ in the direction $v\in\mathbb{R}$:
 $$
  \dmaxx(u;v) = \left\{
      \begin{array}{lr}
        0 & \mbox{ if } \:u<0 ,\\
        \maxx(v) & \mbox{ if } \:u=0, \\
        v & \mbox{ if } \:u>0.
      \end{array}
    \right.
$$
\label{LemDirDiffMax}
\end{lemma}
\begin{lemma}
    The Nemytskij operator $\maxx:L^2(\Gamma_C)\to L^2(\Gamma_C)$ is Lipschitz continuous and directionally differentiable.
\label{LemDirDiffNemytskijMax}
\end{lemma}
The reader is referred to \cite{susu2018optimal}, for example, for the proof of those results.

\begin{notation} 
Let us introduce three subsets of $\mathbb{R}_{+}^{*}\times\mathbb{R}^{d-1}$: 
$$\pazocal{J}^-:=\{ (\alpha,z) \: : \: |z|<\alpha\},
\quad\pazocal{J}^0:=\{ (\alpha,z) \: : \: |z|=\alpha\},\quad 
\pazocal{J}^+:=\{ (\alpha,z) \: : \: |z|>\alpha\},$$
and the functions $\partial_\alpha \qq$ and $\partial_z \qq$, from $\mathbb{R}_+^*\times\mathbb{R}^{d-1}\!\setminus\pazocal{J}^0$ to $\mathcal{L}(\mathbb{R};\mathbb{R}^{d-1})$ and $\mathcal{L}(\mathbb{R}^{d-1})$, respectively, such that:
\begin{equation*}
  \partial_\alpha \qq(\alpha,z) = \left\{
      \begin{array}{lr}
        0 & \mbox{ in } \pazocal{J}^- ,\\
        \frac{z}{|z|} & \mbox{ in } \pazocal{J}^+ ,
      \end{array}
    \right.
\qquad
  \partial_z \qq(\alpha,z) = \left\{
      \begin{array}{lr}
           I_{d-1} & \mbox{in } \pazocal{J}^-, \\
           \frac{\alpha}{|z|}\big(I_{d-1}  -  \frac{1}{|z|^2} z \otimes  z\big) & \mbox{in } \pazocal{J}^+.
      \end{array}
    \right.
\end{equation*}
\end{notation}
\begin{lemma}
 The function $\qq:\mathbb{R}_+^*\times\mathbb{R}^{d-1}\to\mathbb{R}^{d-1}$ is Lipschitz continuous and directionally differentiable, with 
 derivative at $(\alpha,z)$ in the direction $(\beta,h)\in\mathbb{R}\times\mathbb{R}^{d-1}$:
 $$
  \dqq\left((\alpha,z); (\beta,h)\right) = \left\{
      \begin{array}{lr}
        h & \mbox{ in } \pazocal{J}^- ,\\
        h - \maxx\left( h\cdot \frac{z}{|z|}-\beta\right)\frac{z}{|z|} & \mbox{ in } \pazocal{J}^0 , \\
        \frac{\alpha}{|z|}\big(h- \frac{1}{|z|^2} (z\cdot h) z\big) + \beta\frac{z}{|z|} & \mbox{ in } \pazocal{J}^+ .
      \end{array}
    \right.
$$
\label{LemDirDiffQ}
\end{lemma}
\begin{lemma}
    The Nemytskij operator $\qq:L^2(\Gamma_C;\mathbb{R}^*_+)\times\Ll^2(\Gamma_C)\to \Ll^2(\Gamma_C)$ is Lipschitz continuous and directionally differentiable.
\label{LemDirDiffNemytskijQ}
\end{lemma}
\begin{proof}
    First, Lipschitz continuity of this Nemytskij operator follows directly from Lipschitz continuity of $\qq:\mathbb{R}_+^*\times\mathbb{R}^{d-1}\to\mathbb{R}^{d-1}$.
    Then, from \cref{LemDirDiffQ}, it is clear that, for all $(\alpha,z)\in\mathbb{R}_+^*\times\mathbb{R}^{d-1}$ and $(\beta,h)\in\mathbb{R}\times\mathbb{R}^{d-1}$, one has:
    \begin{equation}
         \left| \dqq\left((\alpha,z);(\beta,h)\right)\right| \leq |\beta| + |h|  \:.
         \label{EstPartDerQ}
    \end{equation}
    Let $(\alpha,\zz)\in L^2(\Gamma_C;\mathbb{R}^*_+)\times\Ll^2(\Gamma_C)$ and $(\beta,\hh)\in L^2(\Gamma_C)\times\Ll^2(\Gamma_C)$, and $t>0$. Directional differentiability of $\qq:\mathbb{R}_+^*\times\mathbb{R}^{d-1}\to\mathbb{R}^{d-1}$ yields:
    \begin{equation*}
        \begin{aligned}
            \left|\: \frac{\qq(\alpha+t\beta,\zz+t\hh)-\qq(\alpha,\zz)}{t}- \dqq\left((\alpha,\zz);(\beta,\hh)\right) \:\right| \longrightarrow 0 \hspace{0.5em} \mbox{ a.e.$\!$ on } \Gamma_C.
        \end{aligned}
    \end{equation*}
    Moreover, from estimation \eqref{EstPartDerQ}, along with Lispchitz continuity of $\qq$, 
    one gets:
    \begin{equation*}
        \begin{aligned}
            \left|\: \frac{\qq(\alpha+t\beta,\zz+t\hh)-\qq(\alpha,\zz)}{t}- \dqq\left((\alpha,\zz);(\beta,\hh)\right) \:\right| \leq 2\left(|\beta|+|\hh|\right) \hspace{0.5em} \mbox{ a.e.$\!$ on } \Gamma_C.
        \end{aligned}
    \end{equation*} 
    Since $|\hh|$ and $|\beta|\in L^2(\Gamma_C)$, Lebesgue's dominated convergence theorem finishes the proof.
\end{proof}


\subsubsection{Differentiability of $\Phi_\varepsilon$}


We are now in mesure to state the main results of this work.
\begin{notation}
For any smooth function $f$ defined on $\mathbb{R}^d$, and that does not depend on $\Omega$, we denote $f'[\thetaa]$ or simply $f'$ the following directional derivative:
$$
  f'[\thetaa] := \lim_{t\searrow 0} \:\frac{1}{t}\left( f\circl(\Id+t\thetaa) - f\right) = (\grad f) \thetaa \:.
$$
Using this notation, $\normalExt':=(\gradd \normalExt) \thetaa\:$ and for any $\vv\in\Xx$ we define :
$$
    \vv_{\normalExt'}:=\vv\cdot{\normalExt}'\:,  \hspace{0.5em} \vv_{\tanExt'}:=-\vv\cdot((\gradd \normalExt) \thetaa)\normalExt - (\vv\cdot\normalExt)(\gradd \normalExt) \thetaa = -\vv_{ \normalExt'}\normalExt - \vv_{\normalExt} \normalExt'\:.
$$
For the gap $\,\gG_{\normalExt}':=(\grad \gG_{\normalExt}) \thetaa$ and since $\gG_{\normalExt}$ is the oriented distance function to the smooth boundary $\partial\Omega_{rig}$, $\grad \gG_{\normalExt} = -\normalExt$, which implies that $\gG_{\normalExt}'= -\thetaa\cdot\normalExt$. However, we will still use the notation $\gG_{\normalExt}'$ to emphasize that this term comes from differentiation of the gap. Finally we define 
$$\pazocal{I}_\varepsilon^0:=\{x \in \Gamma_C \: : \: \uu_{\varepsilon,\normalExt}-\gG_{\normalExt}=0\}\subset \Gamma_C\:, \quad \pazocal{J}_\varepsilon^0:=\{x \in \Gamma_C \: : \: |\uu_{\varepsilon,\tanExt}|=\varepsilon\mathfrak{F}s\}\subset \Gamma_C\:,$$ two sets of special interest in the rest of this work.

\end{notation}
\begin{theo} If \cref{A0} holds, then for any $\thetaa\in\Cc^1_b(\mathbb{R}^d)$, $\Phi_\varepsilon$ is strongly differentiable at $t=0^+$.
  \label{ThmExistMDer}
\end{theo}

\begin{proof}
    In order to prove the differentiability of this map, one has to study the difference $\ww_\varepsilon^t:=\frac{1}{t}(\uu_\varepsilon^t-\uu_\varepsilon)$, which appears when dividing \eqref{FVTMoinsFV} by $t$. Of course, this leads to the formulation: $\frac{1}{t}(T_1(\vv)+T_2(\vv)+T_3(\vv))=\frac{1}{t}T_4(\vv)$. 
Again, from \cite[Section 3.5]{sokolowski1992introduction}, taking $\vv=\ww_\varepsilon^t$ as a test-function, one gets the following estimates for the first and fourth groups of terms:
\begin{equation*}
	\begin{aligned}
    \frac{1}{t}\:T_1(\ww_\varepsilon^t) & \: \geq \: \alpha_0 \norml \ww_\varepsilon^t \normr_{\Xx}^2 - C\norml \ww_\varepsilon^t \normr_{\Xx} \:, \\
    \frac{1}{t}\:T_4(\ww_\varepsilon^t) &\: \leq \:  C\norml \ww_\varepsilon^t \normr_{\Xx} \:.
    \end{aligned}
\end{equation*}

Using the property $(\maxx(a)-\maxx(b))(a-b) \geq 0$, for all $a$, $b \in \mathbb{R}$, one gets for the second group of terms:
\begin{equation*}
    \begin{aligned}
        \frac{1}{t}\:T_2(\ww_\varepsilon^t)\:  & = \:\frac{1}{\varepsilon} \int_{\Gamma_{C}} R_{\normalExt}^{t}(\uu_{\varepsilon}^{t}) \ww_{\varepsilon,\normalExt(t)}^t  \: \frac{1}{t}(\JacB(t)-1) 
        \\
        &\quad + \frac{1}{\varepsilon} \int_{\Gamma_{C}} R_{\normalExt}^{t}(\uu_{\varepsilon}^{t}) 
        \frac{1}{t}(\ww_{\varepsilon,\normalExt(t)}^t - \ww_{\varepsilon,\normalExt}^t)
        + \frac{1}{\varepsilon} \int_{\Gamma_{C}} \frac{1}{t} \left( R_{\normalExt}^{t}(\uu_{\varepsilon}^{t}) - 
        R_{\normalExt}(\uu_{\varepsilon})\right) \ww_{\varepsilon,\normalExt}^t  
        \\
        & \geq \: - C\norml \ww_\varepsilon^t \normr_{\Xx} + \:\frac{1}{\varepsilon} \int_{\Gamma_{C}} \frac{1}{t} \left( R_{\normalExt}(\uu_{\varepsilon}^{t}) - 
        R_{\normalExt}(\uu_{\varepsilon})\right)  \ww_{\varepsilon,\normalExt}^t 
        \\
        & \geq \: - C \norml \ww_\varepsilon^t \normr_{\Xx} \:.
    \end{aligned}
\end{equation*}
One can estimate the third term in the same way, using this time the properties of $\qq$, and especially the property $(\qq(\alpha,z_1)-\qq(\alpha,z_2))(z_1-z_2)\geq 0$, for all $\alpha\in\mathbb{R}_+^*$, $z_1,z_2\in\mathbb{R}^{d-1}$.
\begin{equation*}
    \begin{aligned}
        \frac{1}{t}\:T_3(\ww_\varepsilon^t) & =  \: \frac{1}{\varepsilon} \int_{\Gamma_{C}} S_{\tanExt}^t(\uu_{\varepsilon}^{t}) \ww_{\varepsilon,\tanExt(t)}^t \frac{1}{t}(\JacB(t)-1) 
        + \frac{1}{\varepsilon} \int_{\Gamma_{C}} S_{\tanExt}^t(\uu_{\varepsilon}^{t}) \frac{1}{t}(\ww_{\varepsilon,\tanExt(t)}^t-\ww_{\varepsilon,\tanExt}^t) 
        \\
        &\quad + \frac{1}{\varepsilon} \int_{\Gamma_{C}} \frac{1}{t}\left( S_{\tanExt}^t(\uu_{\varepsilon}^{t}) - S_{\tanExt}(\uu_{\varepsilon}) \right) \ww_{\varepsilon,\tanExt}^t  
        \\
        & \geq \: - C\norml \ww_\varepsilon^t \normr_{\Xx} + \frac{1}{\varepsilon} \int_{\Gamma_{C}} \frac{1}{t}\left( S_{\tanExt}(\uu_{\varepsilon}^{t}) - S_{\tanExt}(\uu_{\varepsilon}) \right)\ww_{\varepsilon,\tanExt}^t  
        \\
        & \geq \: - C\norml \ww_\varepsilon^t \normr_{\Xx} \:.
    \end{aligned}
\end{equation*}
Combining these four estimates leads to boundedness of $\ww_\varepsilon^t$ in $\Xx$ (uniformly in $t$). Thus for any sequence $\{t_k\}_k$ decreasing to 0, there exists a weakly convergent subsequence of $\{ \ww_\varepsilon^{t_k} \}_k$ (still denoted $\{ \ww_\varepsilon^{t_k} \}_k$), say $\ww_\varepsilon^{t_k} \rightharpoonup \ww_\varepsilon \in \Xx$. 
\vspace{0.5em}

The next step is to characterize this weak limit as the solution of a variational formulation. This can be done by taking $t=t_k$, then passing to the limit $k\to+\infty$ in formulation \eqref{FVTMoinsFV} divided by $t$. Before doing that, the bilinear form $a'$ and the linear form $\epsilonn'$, which will be very useful, are introduced as in \cite[Section 3.5]{sokolowski1992introduction}: for any $\uu$, $\vv \in \Xx$,
\begin{equation*}
    \begin{aligned} 
        &a'(\uu,\vv) := \int_\Omega \big\{ \Aa:\epsilonn'(\uu):\epsilonn(\vv) + \Aa:\epsilon(\uu):\epsilonn'(\vv) 
        \\ &\hspace{0.25\textwidth} 
        + (\divv \thetaa \: \Aa + \gradd \Aa \: \thetaa):\epsilonn(\uu):\epsilonn(\vv) \big\} \:, \\
        &\epsilonn'(\vv) := -\frac{1}{2}\left( \gradd \vv \gradd \thetaa + {\gradd \thetaa}^T {\gradd \vv}^T \right).
    \end{aligned}
\end{equation*}

Now, passing to the limit $k\to+\infty$ in $T_1(\vv)$ is rather straightforward and gives:
\begin{equation}
    \frac{1}{t_k}\:T_1(\vv) \: \longrightarrow \: a(\ww_\varepsilon,\vv) + a'(\uu_\varepsilon,\vv) \:.
    \label{FVMDerT1}
\end{equation}
For the second group of terms, one gets:
\begin{equation*}
    \begin{aligned}
        \frac{1}{t_k}\:T_2(\vv) \: \longrightarrow \: & \: \frac{1}{\varepsilon} \int_{\Gamma_{C}} R_{\normalExt}(\uu_{\varepsilon}) \left(\vv \cdot (\divv_\Gamma \thetaa \normalExt + \normalExt'\right) \\
        & \qquad + \lim_k \frac{1}{\varepsilon} \int_{\Gamma_{C}} \frac{1}{t_k} \left( R_{\normalExt}^{t_k}(\uu_{\varepsilon}^{t_k})- R_{\normalExt}(\uu_{\varepsilon})\right) (\vv \cdot \normalExt) \:.
    \end{aligned}
\end{equation*}
The key ingredient to deal with the second limit is the directional differentiability of the function $\maxx$ from $L^2(\Gamma_C)$ to $L^2(\Gamma_C)$, see \cref{app:DirDiffQ}. The candidate function for the derivative of $t\mapsto R_{\normalExt}^{t}(\uu_{\varepsilon}^{t})$ at $t=0^+$ is:
$$
    R_{\normalExt}'(\uu_{\varepsilon}):=
    \dmaxx\left( \uu_{\varepsilon,\normalExt}-\gG_{\normalExt} ; \zz_{\varepsilon,\normalExt}^{\thetaa} \right)\:.
$$
where $\zz_{\varepsilon,\normalExt}^{\thetaa} := \ww_{\varepsilon,\normalExt}+\uu_{\varepsilon,\normalExt'}-\gG_{\normalExt}'$. Let us show strong convergence in $L^2(\Gamma_C)$ to this candidate function by estimating:
\begin{equation*}
    \begin{aligned}
        &\norml \frac{1}{t_k} \left( R_{\normalExt}^{t_k}(\uu_{\varepsilon}^{t_k}) - R_{\normalExt}(\uu_{\varepsilon})\right) - R_{\normalExt}'(\uu_{\varepsilon}) \normr_{0,\Gamma_C}  
        \\ & \: 
        \leq \: \left\lVert \frac{1}{t_k} \left( \maxx\left(\uu_{\varepsilon,\normalExt(t_k)}^{t_k}-\gG_{\normalExt}(t_k)\right) -  \maxx\left(\uu_{\varepsilon,\normalExt}-\gG_{\normalExt}+\, t_k\zz_{\varepsilon,\normalExt}^{\thetaa}\right)\right) \right\lVert_{0,\Gamma_C} 
        \\ & \qquad 
        +  \bigg\lVert \frac{1}{t_k} \left( \maxx\left(\uu_{\varepsilon,\normalExt}-\gG_{\normalExt}+\,t_k \zz_{\varepsilon,\normalExt}^{\thetaa} \right)- \maxx\left(\uu_{\varepsilon,\normalExt}-\gG_{\normalExt}\right)\right)
        - R_{\normalExt}'(\uu_{\varepsilon}) \bigg\lVert_{0,\Gamma_C} 
        \\ & \: 
        \leq \: \norml \ww_{\varepsilon}^{t_k} - \ww_{\varepsilon} \normr_{0,\Gamma_C} + t_k C\left( 1 + \norml \ww_\varepsilon \normr_{0,\Gamma_C} + \norml \uu_\varepsilon \normr_{0,\Gamma_C} \right)
        \\ & \quad 
        +  \bigg\lVert \frac{1}{t_k} \left( \maxx\left(\uu_{\varepsilon,\normalExt}-\gG_{\normalExt}+\,t_k \zz_{\varepsilon,\normalExt}^{\thetaa}\right)- \maxx\left(\uu_{\varepsilon,\normalExt}-\gG_{\normalExt}\right)\right)
        - \dmaxx\left( \uu_{\varepsilon,\normalExt}-\gG_{\normalExt} ; \zz_{\varepsilon,\normalExt}^{\thetaa} \right) \bigg\lVert_{0,\Gamma_C}
    \end{aligned}
\end{equation*}
The first term goes to 0 due to compact embedding, and the last one also goes to 0 using directional differentiability of the function $\maxx$ from $L^2(\Gamma_C)$ to $L^2(\Gamma_C)$.
This finally leads for the second group of terms:
\begin{equation}
        \frac{1}{t_k}\:T_2(\vv) \longrightarrow \frac{1}{\varepsilon} \int_{\Gamma_{C}} R_{\normalExt}(\uu_{\varepsilon}) 
        \left(\vv_{\normalExt} \divv_\Gamma \thetaa + \vv_{\normalExt'} \right)
        + \frac{1}{\varepsilon} \int_{\Gamma_C}  R_{\normalExt}'(\uu_{\varepsilon}) \vv_{\normalExt}.
    \label{FVMDerT2}
\end{equation}
From \cref{LemDirDiffMax}, 
$R_{\normalExt}'(\uu_{\varepsilon}) = \maxx(\zz_{\varepsilon,\normalExt}^{\thetaa})$ on $\pazocal{I}_\varepsilon^0$.
The function $\maxx:\mathbb{R}\to\mathbb{R}$ being non-linear, the limit formulation is non-linear in $\thetaa$ if $\pazocal{I}_\varepsilon^0$ is not of null measure.
For the third group of terms, one gets:
\begin{equation*}
    \begin{aligned}
        \frac{1}{t_k}\:T_3(\vv) \: \longrightarrow \: & \: \frac{1}{\varepsilon} \int_{\Gamma_{C}} S_{\tanExt}(\uu_{\varepsilon}) \left(\vv_{\tanExt} \divv_\Gamma \thetaa + \vv_{\tanExt'} \right) \\
        & \qquad + \lim_k \frac{1}{\varepsilon} \int_{\Gamma_{C}} \frac{1}{t_k} \left( S_{\tanExt,t_k}(\uu_{\varepsilon}^{t_k})- S_{\tanExt}(\uu_{\varepsilon})\right) \vv_{\tanExt} \:.
    \end{aligned}
\end{equation*}
The key ingredient is the directional differentiability of the Nemytskij operator associated to $\qq$, see \cref{app:DirDiffQ}. The candidate function for the derivative of $t\mapsto S_{\tanExt}^t(\uu_{\varepsilon}^{t})$ at $t=0^+$ is:
$$
    S_{\tanExt}'(\uu_{\varepsilon}):=\dqq\left( (\varepsilon\mathfrak{F}s,\uu_{\varepsilon,\tanExt}) ; \left(\varepsilon\grad(\mathfrak{F}s)\thetaa, \zz_{\varepsilon,\tanExt}^{\thetaa} \right) \right)\:,
$$
where $\zz_{\varepsilon,\tanExt}^{\thetaa}:=\ww_{\varepsilon,\tanExt}+\uu_{\varepsilon,\tanExt'}$. Another series of estimations gives strong convergence to this candidate function in $\Ll^2(\Gamma_C)$.
\begin{equation*}
    \begin{aligned}
        &\norml \frac{1}{t_k} \left( S_{\tanExt,t_k}(\uu_{\varepsilon}^{t_k}) - S_{\tanExt}(\uu_{\varepsilon})\right) - S_{\tanExt}'(\uu_{\varepsilon}) \normr_{0,\Gamma_C}  
        \\ & 
        \leq \: \left\lVert \frac{1}{t_k} \left( \qq\left( \varepsilon (\mathfrak{F}s)(t_k),\uu_{\varepsilon,\tanExt(t_k)}^{t_k} \right)
        - \qq\left( \varepsilon (\mathfrak{F}s + t_k\grad(\mathfrak{F}s)\thetaa) , \uu_{\varepsilon,\tanExt} + t_k \zz_{\varepsilon,\tanExt}^{\thetaa}\right) \right) \right\lVert_{0,\Gamma_C} 
        \\ & \ \
        +  \bigg\lVert \frac{1}{t_k} \left( \qq\left( \varepsilon (\mathfrak{F}s + t_k\grad(\mathfrak{F}s)\thetaa) , \uu_{\varepsilon,\tanExt} + t_k \zz_{\varepsilon,\tanExt}^{\thetaa}\right)
        - \qq\left( \varepsilon \mathfrak{F}s,\uu_{\varepsilon,\tanExt} \right) \right)
        - S_{\tanExt}'(\uu_{\varepsilon}) \bigg\lVert_{0,\Gamma_C} 
        \\ & 
        \leq \: \norml \ww_{\varepsilon}^{t_k} - \ww_{\varepsilon} \normr_{0,\Gamma_C} + C t_k \left( \varepsilon + \norml \ww_\varepsilon \normr_{0,\Gamma_C} + \norml \uu_\varepsilon \normr_{0,\Gamma_C} \right)
        \\ & \ \
        + \bigg\lVert \frac{1}{t_k} \left( \qq\left( \varepsilon (\mathfrak{F}s + t_k\grad(\mathfrak{F}s)\thetaa) , \uu_{\varepsilon,\tanExt} + t_k \zz_{\varepsilon,\tanExt}^{\thetaa}\right)
        - \qq\left( \varepsilon \mathfrak{F}s,\uu_{\varepsilon,\tanExt} \right) \right)
        - S_{\tanExt}'(\uu_{\varepsilon}) \bigg\lVert_{0,\Gamma_C} \!.
    \end{aligned}
\end{equation*}
Due to compact embedding and directional differentiability 
for $\qq$, all terms on the right hand side converge to 0.  Thus, 
\begin{equation}
    \begin{aligned}
        \frac{1}{t_k}\:T_3(\vv) \: \longrightarrow \: & \: \frac{1}{\varepsilon} \int_{\Gamma_{C}} S_{\tanExt}(\uu_{\varepsilon}) \left(\vv_{\tanExt} \divv_\Gamma \thetaa + \vv_{\tanExt'} \right) 
        + \frac{1}{\varepsilon} \int_{\Gamma_{C}} S_{\tanExt}'(\uu_{\varepsilon}) \vv_{\tanExt} \:.
    \end{aligned}
    \label{FVMDerT3}
\end{equation}
From \cref{LemDirDiffQ}, $S_{\tanExt}'(\uu_{\varepsilon})$ is non linear uniquely on $\pazocal{J}_\varepsilon^0$ where it uses the $\maxx$ function. Therefore the limit variational formulation is non linear only on $ \pazocal{J}_\varepsilon^0$.

Using once again the results from \cite[Section 3.5]{sokolowski1992introduction} gives 
\begin{equation}
    \frac{1}{t_k}\:T_4(\vv) \: \longrightarrow \: \int_\Omega \left(\divv \thetaa \: \ff + \gradd\ff \thetaa\right)\vv + \int_{\Gamma_N} \left(\divv_\Gamma \thetaa \: \tauu + \gradd\tauu \thetaa\right) \vv \:.
    \label{FVMDerT4}
\end{equation}
Combining \eqref{FVMDerT1}, \eqref{FVMDerT2}, \eqref{FVMDerT3} and \eqref{FVMDerT4}, and using the Heaviside function $H$ and $\partial_\alpha$ and $\partial_z$ defined  
in \cref{app:DirDiffQ}, one gets that $\ww_\varepsilon\in \Xx$ is the solution of
\begin{equation}
        b_\varepsilon(\ww_\varepsilon,\vv) + \frac{1}{\varepsilon} \prodL2{  R_{\normalExt}'(\uu_{\varepsilon}), \vv_{\normalExt} }{\pazocal{I}_\varepsilon^0} 
        + \frac{1}{\varepsilon} \prodL2{ S_{\tanExt}'(\uu_{\varepsilon}) , \vv_{\tanExt} }{\pazocal{J}_\varepsilon^0} = L_\varepsilon[\thetaa](\vv) \:, \hspace{1em} 
        \forall \vv \in \Xx ,
    \label{FVMDer}
\end{equation} 
where the bilinear form $b_\varepsilon$ and linear form $L_\varepsilon[\thetaa]$ are defined as, for any $\uu$, $\vv\in\Xx$:
\begin{equation*}
    \begin{aligned}
        b_\varepsilon(\uu,\vv):=a(\uu,\vv) &+ \frac{1}{\varepsilon} \prodL2{H(\uu_{\varepsilon,\normalExt}-\gG_{\normalExt}) \uu_{\normalExt} , \vv_{\normalExt}}{\Gamma_C\setminus\pazocal{I}_\varepsilon^0 } \\
        \:& + \frac{1}{\varepsilon} \prodL2{\partial_z \qq(\varepsilon\mathfrak{F}s,\uu_{\varepsilon,\tanExt}) \uu_{\tanExt} , \vv_{\tanExt}}{\Gamma_C\setminus\pazocal{J}_\varepsilon^0} ,
    \end{aligned}
\end{equation*}
\begin{equation*}
    \begin{aligned}
        L_\varepsilon[\thetaa](\vv) := &\int_\Omega (\divv \thetaa \: \ff + \gradd \ff  \thetaa) \vv + \int_{\Gamma_N} (\divv_\Gamma \thetaa \: \tauu+ \gradd \tauu \thetaa)\vv - \:a'(\uu_\varepsilon,\vv) \\
        & - \frac{1}{\varepsilon} \int_{\Gamma_C} R_{\normalExt}(\uu_{\varepsilon})\left(\vv \cdot (\divv_\Gamma \thetaa \normalExt + \normalExt')\right) \\
        & - \frac{1}{\varepsilon} \int_{\Gamma_C\setminus\pazocal{I}_\varepsilon^0} H(\uu_{\varepsilon,\normalExt}-\gG_{\normalExt}) \left( \uu_{\varepsilon,\normalExt'} - \gG_{\normalExt}' \right)\vv_{\normalExt} \\
        & - \frac{1}{\varepsilon} \int_{\Gamma_{C}} S_{\tanExt}(\uu_{\varepsilon}) \left(\vv_{\tanExt} \divv_\Gamma \thetaa + \vv_{\tanExt'} \right) \\
        & - \int_{\Gamma_{C}\setminus\pazocal{J}_\varepsilon^0} \left( \partial_\alpha \qq (\varepsilon\mathfrak{F}s,\uu_{\varepsilon,\tanExt}) \grad(\mathfrak{F}s)\thetaa 
        + 
        \frac{1}{\varepsilon} \partial_z \qq (\varepsilon\mathfrak{F}s,\uu_{\varepsilon,\tanExt}) \uu_{\varepsilon,\tanExt'}\right) \vv_{\tanExt}\:.
    \end{aligned}
\end{equation*}
Due to the regularities of $\normalExt$, $\gG_{\normalExt}$, $\thetaa$, $\ff$, $\tauu$, $\mathfrak{F}s$, $\uu_\varepsilon$, and uniform boundedness of both $\partial_\alpha\qq$, $\partial_z\qq$, it is clear that $L_\varepsilon[\thetaa] \in \Xx^*$ for any $\thetaa$. 
From the uniform boundedness and positivity of $H(\cdot)$ and $\partial_z \qq(\cdot,\cdot)$, one has, for all $\uu$, $\vv$, in $\Xx$
\begin{equation*}
  \begin{aligned}
   \left| \frac{1}{\varepsilon} \prodL2{H(\uu_{\varepsilon,\normalExt}-\gG_{\normalExt}) \uu_{\normalExt} , \vv_{\normalExt}}{\Gamma_C\setminus\pazocal{I}_\varepsilon^0} \right| &\leq \frac{K}{\varepsilon}\norml \uu \normr_{\Xx} \norml \vv \normr_{\Xx} \:, \\
   \frac{1}{\varepsilon} \prodL2{H(\uu_{\varepsilon,\normalExt}-\gG_{\normalExt}) \uu_{\normalExt} , \uu_{\normalExt}}{\Gamma_C\setminus\pazocal{I}_\varepsilon^0} &= \frac{1}{\varepsilon} \int_{\Gamma_C\setminus\pazocal{I}_\varepsilon^0} H(\uu_{\varepsilon,\normalExt}-\gG_{\normalExt}) \:(\uu_{\normalExt})^2 \geq 0 \: , 
   \\
   \left| \frac{1}{\varepsilon} \prodL2{\partial_z\qq(\varepsilon\mathfrak{F}s,\uu_{\varepsilon,\tanExt}) \uu_{\tanExt} , \vv_{\tanExt}}{\Gamma_C\setminus\pazocal{J}_\varepsilon^0} \right| 
   & \leq \frac{K}{\varepsilon}\norml \uu \normr_{\Xx} \norml \vv \normr_{\Xx} \:, 
   \\
   \frac{1}{\varepsilon} \prodL2{\partial_z\qq(\varepsilon\mathfrak{F}s,\uu_{\varepsilon,\tanExt}) \uu_{\tanExt} , \uu_{\tanExt}}{\Gamma_C\setminus\pazocal{J}_\varepsilon^0} &= \frac{1}{\varepsilon} \int_{\Gamma_C\setminus\pazocal{J}_\varepsilon^0}  \left(\partial_z\qq(\varepsilon\mathfrak{F}s,\uu_{\varepsilon,\tanExt}) \uu_{\tanExt}\right) \uu_{\tanExt} \geq 0 \: .
   \end{aligned}
\end{equation*} 
Thus  $b_\varepsilon$ is continuous and coercive over $\Xx \times \Xx$. Because of the non-linearities occuring on the sets $\pazocal{I}_\varepsilon^0$ and $\pazocal{J}_\varepsilon^0$, well-posedness of \eqref{FVMDer} is proved using an optimization argument. Let us introduce the following functionals, defined for any $\ww\in\Xx$:
$$
   \begin{aligned}
   \tilde{\phi}(\ww) := &\frac{1}{2} b_\varepsilon(\ww,\ww) - L_\varepsilon[\thetaa](\ww) + \phi_{\normalExt}(\ww) + \phi_{\tanExt}(\ww) \:,\\[1em]
        \phi_{\normalExt}(\ww) := \ & \frac{1}{2\varepsilon} \norml \maxx\left( \ww_{\normalExt}+\uu_{\varepsilon,\normalExt'}-\gG_{\normalExt}' \right) \normr^2_{0,\pazocal{I}_\varepsilon^0} , \\
        \phi_{\tanExt}(\ww) := \ & 
        \frac{1}{2\varepsilon} \norml \ww_{\tanExt}+\uu_{\varepsilon,\tanExt'} \normr^2_{0,\pazocal{J}_\varepsilon^0} \\ 
        & + \frac{1}{2\varepsilon} \norml 
        \maxx\left(
        -\varepsilon\grad(\mathfrak{F}s)\thetaa + \left(\ww_{\tanExt}+\uu_{\varepsilon,\tanExt'}\right)\cdot \frac{ \uu_{\varepsilon,\tanExt}}{| \uu_{\varepsilon,\tanExt}|}
        \right)
        \normr^2_{0,\pazocal{J}_\varepsilon^0} \:.
   \end{aligned}
$$
Obviously, solving \eqref{FVMDer} is equivalent to finding a minimum of $\tilde{\phi}$ over $\Xx$. Both $\phi_{\normalExt}$ and $\phi_{\tanExt}$ are convex, continuous and positive.   Due to the properties of $b_\varepsilon$ and $L_\varepsilon[\thetaa]$, $\tilde{\phi}$ is strictly convex, coercive and continuous, and one gets that problem \eqref{FVMDer} has a unique solution $\ww_\varepsilon$. Uniqueness also proves that the whole sequence $\{ \ww_\varepsilon^{t_k}\}_k$ converges weakly to $\ww_\varepsilon$.

\paragraph{Strong convergence.} 
Strong convergence is proved taking $\boldsymbol{\delta}_{\ww,\varepsilon}^t:=\ww_\varepsilon^t-\ww_\varepsilon$ as test-function and subtracting: $\frac{1}{t}$\eqref{FVTMoinsFV} $-$ \eqref{FVMDer}. Gathering all terms properly enable to get the following estimation:
\begin{equation*}
    \begin{aligned}
        \alpha_0 \norml \boldsymbol{\delta}_{\ww,\varepsilon}^t \normr_{\Xx}^2 \ \leq & \ \left( tC + \norml \frac{1}{t} \left( R_{\normalExt}^{t}(\uu_{\varepsilon}^{t}) - R_{\normalExt}(\uu_{\varepsilon})\right) - R_{\normalExt}'(\uu_{\varepsilon}) \normr_{0,\Gamma_C} \right.\\
        & \ \ \left. + \norml \frac{1}{t} \left( S_{\tanExt}^t(\uu_{\varepsilon}^{t}) - S_{\tanExt}(\uu_{\varepsilon})\right) - S_{\tanExt}'(\uu_{\varepsilon}) \normr_{0,\Gamma_C}\right) \norml \boldsymbol{\delta}_{\ww,\varepsilon}^t \normr_{\Xx} \:.
    \end{aligned}
\end{equation*}
It has already been showed that all terms in parentheses go to $0$, which yields strong convergence of $\ww_\varepsilon^t$ to $\ww_\varepsilon$ in $\Xx$.
\end{proof}

Existence and uniqueness of the limit $\frac{1}{t}(\uu_\varepsilon^t-\uu_\varepsilon)$ have been established. In other words, it has been proved that $\uu_\varepsilon$ admits a strong material derivative in any direction $\thetaa$, namely $\ww_\varepsilon = \dot{\uu}_\varepsilon(\Omega)[\thetaa] \in \Xx$, or simply $\ww_\varepsilon = \dot{\uu}_\varepsilon\in \Xx$. Nevertheless, as mentionned in the previous proof, the map $\thetaa\mapsto \dot{\uu}_\varepsilon(\Omega)[\thetaa]$ fails to be linear  on $\pazocal{I}_\varepsilon^0 \cup \pazocal{J}_\varepsilon^0$ due to non-Gâteaux-differentiability of $\maxx$ and $\qq$. Thus, some additional assumptions are required.
%
%
A rather straightforward way to get around this is to assume that for a fixed value of $\varepsilon$, those sets are of measure zero: 

\begin{hypothesis}\label{A1}
    The sets $\pazocal{I}_\varepsilon^0$ and $\pazocal{J}_\varepsilon^0$ are of measure 0.
\end{hypothesis}

Note that, due to \eqref{CL2:1}, $x\in \pazocal{I}_\varepsilon^0$ implies that both $\uu_{\varepsilon,\normalExt}(x)-\gG_{\normalExt}(x)=0$ and $\sigmaa_{\normalInt\normalExt}(\uu_\varepsilon) (x)=0$, which means that $x$ is in contact but there is no contact pressure. On the other hand, by definition, a point $x\in \pazocal{J}_\varepsilon^0$ is at the same time in sliding contact and in sticking contact. In the case of the non-penalty formulation, $\pazocal{I}_\varepsilon^0$ is sometimes referred to as the \textit{weak contact set}, while $\pazocal{J}_\varepsilon^0$ is sometimes referred to as the \textit{weak sticking set} (see~\cite{beremlijski2014shape} for contact with Coulomb friction). Following these denominations, let us refer to the points of  $\pazocal{I}_\varepsilon^0$ and $\pazocal{J}_\varepsilon^0$ as \textit{weak contact points}, and \textit{weak sticking points}, respectively.

For example, \cref{A1} is satisfied when all weak contact points and all weak sticking points represent a finite number of points in 2D or a finite number of curves in 3D.

\begin{remark}
    Both sets can be gathered under the more general denomination of \textit{biactive sets}, borrowed from optimal control (see~\cite{wachsmuth2014strong} in the case of the obstacle problem). Moreover, in optimal control problems related to variational inequalities, Gâteaux differentiability of the solution with respect to the control parameter is obtained under the \textit{strict complementarity condition}, see for example \cite{bonnans2013perturbation}. This condition is actually quite difficult to explicit and to use in practice, see \cite[Lemma 2.6]{rauls2018generalized}, and \cite{wachsmuth2014strong} for a discussion. However, in our context, the variational inequality has  been regularized by the penalty approach. Therefore our conditions are simpler to express: the biactive sets are of zero measure.
\end{remark}

\begin{cor}  \label{CorExistSDer}
  If \cref{A0} and \cref{A1} hold, then $\uu_\varepsilon$ solution of \eqref{FV} is (strongly) shape differentiable in $\Ll^2(\Omega)$. For any $\thetaa\in\Cc^1_b(\mathbb{R}^d)$, its shape derivative in the direction $\thetaa$ writes $d\mathbf{u}_\varepsilon(\Omega)[\thetaa] :=  \dot{\uu}_\varepsilon(\Omega)[\thetaa] - \gradd \uu_\varepsilon \thetaa$, where $\dot{\uu}_\varepsilon(\Omega)[\thetaa]$ is the unique solution of 
    \begin{equation}
        b_\varepsilon(\dot{\uu}_\varepsilon,\vv) = L_\varepsilon[\thetaa](\vv) \:, \hspace{1em} \forall \vv \in \Xx .
        \label{FVMDerBis}
    \end{equation}
    Moreover, $\{\dot{\uu}_\varepsilon\}_\varepsilon$ and $\{d\mathbf{u}_\varepsilon\}_\varepsilon$ are uniformly bounded in $\Xx$ and $\Ll^2(\Omega)$, respectively.
\end{cor}

\begin{proof}
    When \cref{A1} holds, the variational formulation \eqref{FVMDer} solved by $\dot{\uu}_\varepsilon$ may be rewritten as \eqref{FVMDerBis}.
    Since the map $\thetaa\mapsto L_\varepsilon[\thetaa]$ is linear from $\Cc^1_b(\mathbb{R}^d)$ to $\Xx^*$, one gets that the map $\thetaa \mapsto \dot{\uu}_\varepsilon(\Omega)[\thetaa] \in \Xx$ is linear as well, which directly leads to the desired result. 
    
    Regarding boundedness of $\{\dot{\uu}_\varepsilon\}_\varepsilon$, the key ingredient is the choice of the right test function. Let us introduce
    $$
      \tilde{\uu}_\varepsilon := \left( \uu_{\varepsilon,\normalExt'}-\gG_{\normalExt}' \right)\normalExt - \uu_{\varepsilon,\normalExt}\normalExt'\:.
    $$
    It is clear that $\tilde{\uu}_\varepsilon\in\Xx$, and that one has the following estimation
    $$
      \norml \tilde{\uu}_\varepsilon \normr_{\Xx} \leq C\left( 1 + \norml \uu_\varepsilon \normr_{\Xx} \right) \:.
    $$
    Now, as $\normalExt' \perp \normalExt$, if $\vv\in\Xx$ is defined by $\vv = \dot{\uu}_\varepsilon + \tilde{\uu}_\varepsilon$, then 
    $$
        \vv_{\normalExt} =  \dot{\uu}_{\varepsilon,\normalExt} + \uu_{\varepsilon,\normalExt'}-\gG_{\normalExt}' \:, \quad\quad
        \vv_{\tanExt} =  \dot{\uu}_{\varepsilon,\tanExt} - \uu_{\varepsilon,\normalExt}\normalExt'\:.
    $$
    Therefore, due to positivity of $H$ and $\partial_z \qq$, combined with uniform boundedness of both $\frac{1}{\varepsilon}R_{\normalExt}(\uu_\varepsilon)$ in $L^2(\Gamma_C)$ and $\frac{1}{\varepsilon}S_{\tanExt}(\uu_\varepsilon)$ in $\Ll^2(\Gamma_C)$, taking such a $\vv$ as test-function in \eqref{FVMDerBis} enables to conclude.
\end{proof}

\begin{remark}
    Another approach to get around this non-differentiability issue is to modify the formulation by regularizing non-smooth functions: in this case, replacing $\maxx$ and $\qq$ by regularized versions $\text{p}_{c,+}$ and $\qq_c$, where $c$ stands for the regularization parameter, $c\to\infty$. This leads to a solution map that is Fréchet-differentiable. It can be proved, see \cite{chaudet2019phd}, that the solution of the regularized formulation $\uu_\varepsilon^c \to \uu_\varepsilon$ in $\Xx$, and that in addition, when \cref{A1} holds, the shape derivative $d\mathbf{u}_\varepsilon^c \to d\mathbf{u}_\varepsilon$ in $\Ll^2(\Omega)$.
\end{remark}

\begin{remark}
	Uniform boundedness of $\{d\mathbf{u}_\varepsilon\}_\varepsilon$ implies that the sequence converges weakly in $\Ll^2(\Omega)$ (up to a subsequence) when $\varepsilon\to 0$. However, it seems difficult to characterize this weak limit.
\end{remark}

\subsection{Computation of the shape derivative of a general criterion}

Now that shape sensitivity of the penalty formulation have been studied, one may go back to our initial shape optimization problem \eqref{ShapeOPT}. Let us focus on cost functionals of the rather general type:
\begin{equation}
  J_\varepsilon(\Omega) := \int_\Omega j(\uu_\varepsilon(\Omega)) + \int_{\partial\Omega} k(\uu_\varepsilon(\Omega))\:,
  \label{GeneralJType}
\end{equation}
where $\uu_\varepsilon(\Omega)$ is the solution of \eqref{FV} on $\Omega$. The functions $j,k$ are $\pazocal{C}^1(\mathbb{R}^d,\mathbb{R})$, and their derivatives with respect to $\uu_\varepsilon$, denoted $j'$, $k'$, are Lipschitz. It is also assumed that those functions and their derivatives satisfy, for all $u$, $v \in \mathbb{R}^d$,
\begin{equation}
    |j(u)| \leq C\left(1+|u|^2\right) \qquad
    |k(u)| \leq C\left(1+|u|^2\right)
  \label{Condjk}
\end{equation}
\begin{equation}
    |j'(u)\cdot v| \leq C |u\cdot v| \qquad
    |k'(u)\cdot v| \leq C|u\cdot v| 
  \label{Condj'k'}
\end{equation}
for some constants $C>0$.
From the shape differentiability of $\uu_\varepsilon$, one may deduce the following results, see for example \cite{henrot2006variation}.

\begin{theo}  \label{ThmDJVol}
    When \cref{A0} and \cref{A1} hold, 
    $J_\varepsilon$ is defined by \eqref{GeneralJType} and satisfy \eqref{Condjk} and \eqref{Condj'k'}
    and $\uu_\varepsilon$ is the solution of \eqref{FV}, then 
    $J_\varepsilon$ 
    is shape differentiable at $\Omega$ and its derivative in the direction $\thetaa\in\Cc^1_b(\mathbb{R}^d)$ writes:
    \begin{equation}
            dJ_\varepsilon(\Omega)[\thetaa] = \int_\Omega j'(\uu_\varepsilon)\cdot \dot{\uu}_\varepsilon + \: j(\uu_\varepsilon)\divv\thetaa 
            + \int_{\partial\Omega} k'(\uu_\varepsilon)\cdot \dot{\uu}_\varepsilon + k(\uu_\varepsilon) \divv_\Gamma\thetaa.
        \label{DJVol0}
    \end{equation}
    with $\dot{\uu}_\varepsilon$ solution of \eqref{FVMDerBis}.
\end{theo}

\begin{remark}
  From \cref{CorExistSDer}, one automatically gets that $dJ_\varepsilon$ is uniformly bounded in $\varepsilon$. Therefore, formula \eqref{DJVol0} produces usable shape derivatives, regardless how small $\varepsilon$ gets.    
\end{remark}


From a numerical point of view, this last expression contains a number of difficulties (mainly the right hand side of \eqref{FVMDerBis} and divergence of $\thetaa$) that can be circumvented through simple transformations. 
Introducing the adjoint state, it is possible to rewrite \eqref{DJVol0} avoiding the construction of the right hand side of \eqref{FVMDerBis} and resulting in an expression having only boundary integrals with integrand involving only $\thetaa$.

In the context of problem \eqref{FV} with the functional $J_\varepsilon$, the associated \textit{adjoint state} $\pp_\varepsilon \in \Xx$ is defined as the solution of: 
\begin{equation}
    b_\varepsilon(\pp_\varepsilon,\vv) = -\int_\Omega j'(\uu_\varepsilon)\cdot \vv - \int_{\partial\Omega} k'(\uu_\varepsilon)\cdot \vv \qquad \forall\vv \in \Xx\:.
    \label{FVA}
\end{equation}
Note that by application of Lax-Milgram lemma, existence and uniqueness of $\pp_\varepsilon$ are guaranteed. Using this adjoint state, one is able to get a boundary integral expression of the following form for $dJ_\varepsilon(\Omega)[\thetaa]$.
\begin{theo}
  Suppose $\Omega$ is of class $\pazocal{C}^2$. Then, under the hypothesis of \cref{ThmDJVol}, with $\uu_\varepsilon$, $\pp_\varepsilon\in \Hh^2(\Omega)\cap\Xx$ solutions of \eqref{FV} and \eqref{FVA} respectively,  one has:
  \begin{equation}
      dJ_\varepsilon(\Omega)[\thetaa] = \int_{\partial \Omega} \mathfrak{A}_\varepsilon \: (\thetaa\cdot\normalInt) + \int_{\Gamma_N} \mathfrak{B}_\varepsilon \: (\thetaa\cdot\normalInt) + \int_{\Gamma_C} \mathfrak{C}_\varepsilon \: (\thetaa\cdot\normalInt) \:,
    \label{DJ}
  \end{equation} 
  where $\mathfrak{A}_\varepsilon$, $\mathfrak{B}_\varepsilon$ and $\mathfrak{C}_\varepsilon$ depend on $\uu_\varepsilon$, $\pp_\varepsilon$, their gradients, and the data.
  \label{ThmDJ}
\end{theo}
\begin{proof}
  Due to \cref{ThmDJVol}, when considering \eqref{FVA} with $\vv=\dot{\uu}_\varepsilon\in\Xx$ as test-function, one gets:
  $$
    dJ_\varepsilon(\Omega)[\thetaa] = -b_\varepsilon(\pp_\varepsilon,\dot{\uu}_\varepsilon) + \int_\Omega j(\uu_\varepsilon)\divv\thetaa + \int_{\partial\Omega} k(\uu_\varepsilon) \divv_\Gamma\thetaa \:.
  $$
  Now, noting that $b_\varepsilon$ is symmetric and taking $\vv=\pp_\varepsilon\in\Xx$ in \eqref{FVMDer} leads to
  \begin{equation}
    dJ_\varepsilon(\Omega)[\thetaa] = -L_\varepsilon[\thetaa](\pp_\varepsilon) + \int_\Omega j(\uu_\varepsilon)\divv\thetaa + \int_{\partial\Omega} k(\uu_\varepsilon) \divv_\Gamma\thetaa \:.
    \label{DJVol}
  \end{equation}
  From that point, due to the additional regularity assumption on $\uu_\varepsilon$ and $\pp_\varepsilon$, integrating by parts and using the variational formulations \eqref{FV} and \eqref{FVA} with well chosen test-functions yields the desired result, with
  \begin{equation}\label{EqABC}
  \left\{
      \begin{aligned}
      \mathfrak{A}_\varepsilon &= j(\uu_\varepsilon) + (\kappa +\partial_{\normalInt})k(\uu_\varepsilon) + \Aa:\epsilonn(\uu_\varepsilon):\epsilonn(\pp_\varepsilon) - \ff\pp_\varepsilon \:,\\[0.5em]
      \mathfrak{B}_\varepsilon &= -(\kappa +\partial_{\normalInt})\left( \tauu \pp_\varepsilon \right) \:,\\
      \mathfrak{C}_\varepsilon &=  \mathfrak{C}_{\varepsilon}^{\normalExt} + \mathfrak{C}_{\varepsilon}^{\tanExt} = \frac{1}{\varepsilon}(\kappa +\partial_{\normalInt})\left(R_{\normalExt}(\uu_\varepsilon)\pp_{\varepsilon, \normalExt}\right) +\frac{1}{\varepsilon}(\kappa +\partial_{\normalInt})\left(S_{\tanExt}(\uu_\varepsilon)\pp_{\varepsilon,\tanExt}\right) \:.
  \end{aligned}
  \right.
  \end{equation}
  In the previous formulae, $\kappa$ denotes the mean curvature on $\partial\Omega$, and $\partial_{\normalInt}$ stands for the normal derivative with respect to $\normalInt$.
\end{proof}

\begin{remark}\label{RemNormalDefField}
  Expression \eqref{DJVol} is often referred to as the \textit{distributed shape deri\-vative}, and it is always valid as it only requires $\uu_\varepsilon$, $\pp_\varepsilon\in\Xx$. The reader is referred to \cite{hiptmair2015comparison,laurain2016distributed} for more details about distributed shape derivatives. The additional regularity assumption enables to get an explicit expression that fits the Hadamard-Zolésio structure theorem. This structure of the shape derivative suggests to consider deformation fields $\thetaa$ of the form $\thetaa=\theta \normalInt$, where the normal $\normalInt$ has been extended to $\mathbb{R}^d$ (not necessarily using the oriented distance function to $\partial\Omega$), which is possible when $\partial\Omega$ is at least $\pazocal{C}^1$, see \cite{henrot2006variation}.
\end{remark}

\begin{remark}
  The first two terms in \eqref{DJ} are exactly the same as for the elasticity formulation without contact. There are two additional components coming from the contact conditions, namely $\mathfrak{C}_\varepsilon^{\normalExt}$, stemming from the normal constraint, and $\mathfrak{C}_\varepsilon^{\tanExt}$, stemming from the tangential constraint. Obviously, these are the only terms involving $\normalExt$. When considering problems without contact, those last two terms cancel, while in the case of pure sliding contact, only $\mathfrak{C}_\varepsilon^{\tanExt}\equiv 0$. As for contact problems without gap (see for instance \cite{maury2017shape}), in the expression of $R_{\normalExt}$ the gap $\gG_{\normalExt}$ is simply set to 0, and \eqref{DJ}, \eqref{EqABC} coincides with the derivative in \cite{maury2017shape}. Moreover, note that neglecting the term with $\mathfrak{C}_\varepsilon$ (imposing $\thetaa=0$ on $\Gamma_C$) is equivalent to excluding the contact zone from the optimization process. In other words, the derived expression \eqref{EqABC} is rather general and adapts to many situations: sliding or frictional contact, contact with or without gap, optimizing or not the contact zone, etc.
\end{remark}
%

  We decided to write the contact boundary conditions using the normal $\normalExt$ to the rigid foundation instead of the normal $\normalInt$ to $\partial\Omega$ because it leads to a simpler expression for $dJ_\varepsilon$. Indeed, when differentiating our formulation with respect to the shape, as $\normalExt$ and $\gG_{\normalExt}$ do not depend on $\Omega$, the contact boundary condition can be treated like any Neumann condition. Alternatively, when differentiating the classical formulation, based on $\normalInt$ for the contact boundary conditions, additional terms involving the shape derivatives of the gap, $d\mathbf{g}_{\normalInt}$, and the normal, $d\mathbf{n_o}$, appear in the shape derivative of $J_\varepsilon$ (see \cite{maury2017shape} in the case with no gap). It turns out that these shape derivatives are quite technical to handle in practice. As these two formulations solve the same mechanical problem, see \cref{RemHypHPP}, the simplified expression for $dJ_\varepsilon$, \eqref{DJ} with \eqref{EqABC}, is valid for both formulations.

\section{Numerical results}

\subsection{Shape optimization algorithm}
Following the usual approach, 
\cite{allaire2004structural, maury2017shape},
the algorithm proposed here to minimize $J_\varepsilon(\Omega)$ is a descent method, based on the shape derivative. 
Starting from an initial shape $\Omega^0\subset D$, using the cost functional derivative \eqref{DJ}, the algorithm generates a sequence of shapes $\Omega^k\in \pazocal{U}_{ad}$ such that the real-valued sequence $\{J_\varepsilon(\Omega^k)\}_k$ decreases. Each shape $\Omega^k$ is represented explictly, as a meshed subdomain of $D$, as well as implicitly, as the zero level set of some function $\phi^k$. The explicit representation enables to apply all boundary conditions rigorously, while the implicit representation enables to make the shape evolve smoothly from an iteration to the next by solving the following Hamilton-Jacobi equation on $[0,T]\times \mathbb{R}^d$:
\begin{equation}
    \begin{aligned}
    &\frac{\partial\phi}{\partial t} + \theta |\grad\phi| = 0\:, \\
    &\phi(0,x) = \tilde{\phi}(x)\:,
    \label{HamJac}
    \end{aligned}
\end{equation}
where $T$ is strictly positive, $\tilde{\phi}$ is a given initial condition,  and $\theta$ is the norm of the normal deformation field (see \cref{RemNormalDefField}). The reader is referred to the pioneer work \cite{allaire2004structural} for more details about shape optimization using the level set method.
\vspace{1em}

As mentioned earlier, the method to generate the sequence $\{\Omega^k\}_k$ is based on a gradient descent. It consists in several successive steps.
\begin{enumerate}[leftmargin=3em, topsep=4pt]
    \item Find $\uu_\varepsilon^{k}$ solution of \eqref{FV} on $\Omega^k$.
    \item Find the adjoint state $\pp_\varepsilon^{k}$ solution of \eqref{FVA} on $\Omega^k$.
    \item Find a descent direction with $\thetaa^k=\theta^k \normalInt^k$ using \eqref{DJ} and \eqref{EqABC}.
    \item Update the level set function $\phi^{k+1}$ by solving \eqref{HamJac} on some interval $[0,T^k]$ with $T^k>0$, taking $\theta=\theta^k$ as velocity field and $\tilde{\phi}=\phi^k$ as initial condition.
    \item Cut the mesh of $D$ around $\{\phi^{k+1}\!=\!0\}$ to get an explicit representation of $\Omega^{k+1}$.
\end{enumerate}

\begin{remark}
    In step 4, the real number $T^k$ is chosen such that the monotonicity of $\{J_\varepsilon(\Omega^k)\}_k$ is guaranteed at each iteration. This numerical trick tries to ensure a descent direction, even in situations where \cref{A1} is not verified. Indeed, in such situations, expression \eqref{DJ} will not be an accurate representation of the shape derivative, however it still can provide a valid descent direction.
\end{remark}

\begin{remark}
    Even though such algorithms prove themselves very efficient from the numerical point of view, there are a few limitations. First, note that, in general, problem $\eqref{ShapeOPT}$ is not well-posed and $J_\varepsilon$ is not convex. Thus, using a descent method to try and solve it necessarily leads to finding a local minimum that is highly dependant on the initial shape $\Omega^0$. The reader is referred to \cite{allaire2004structural} for a more detailed discussion on that matter. Second, the algorithm may generate shapes for which the expression of $dJ_\varepsilon$ is inaccurate (e.g. \cref{A1} is not verified) and, in the worst case scenario, from which no descent direction $\thetaa^k$ can be obtained. In such cases, which have not been encountered in practice, the algorithm will stop and no solution will be found.
\end{remark}


\paragraph{Some details about the implementation.}

Although it is not the purpose of this work, we present summarily some aspects of the implementation. Numerical experiments are performed with the code MEF++, developped at the GIREF (Groupe Interdisciplinaire de Recherche en \'El\'ements Finis, Université Laval). Problems \eqref{FV} and \eqref{FVA} are solved using the finite element method using Lagrange $P^2$ finite elements. The Hamilton-Jacobi type equation is solved on a secondary grid, using the second order finite difference scheme presented in \cite{OshSet1988}, with Neumann boundary conditions on $\partial D$. The reader is referred to the rather recent work \cite{chouly2013convergence} for finite element resolution and error estimate of the penalty formulation of contact problems in linear elasticity, and to \cite{Set1996,QuaVal2008} for details about level set methods and their numerical treatment using finite differences.


\subsection{Specific context}

Even though the method could deal with any functional $J_\varepsilon$ of the general type \eqref{GeneralJType}, we focus here on the special case of a linear combination of the compliance and the volume (with some weight coefficients $\alpha_1$ and $\alpha_2$).
\begin{equation*}
    J_\varepsilon(\Omega) = \alpha_1 C(\Omega) + \alpha_2 \textit{Vol}(\Omega)= \int_\Omega (\alpha_1\ff \uu_\varepsilon(\Omega)+\alpha_2)+ \int_{\Gamma_N} \alpha_1\tauu \uu_\varepsilon(\Omega) \:.
\end{equation*}
Indeed, from the engineering point of view, minimizing such a $J_\varepsilon$ means finding the best compromise (in some sense) between weight and stiffness. 

The materials are assumed to be isotropic and obeying Hooke's law (linear elastic), that is:
$$
    \sigmaa(\uu) = \Aa:\epsilonn(\uu) = 2\mu \epsilonn(\uu)+\lambda \divv\uu \:, 
$$
where $\lambda$ and $\mu$ are the Lamé coefficients of the material, which can be expressed in terms of Young's modulus $E$ and Poisson's ratio $\nu$:
\begin{equation*}
    \lambda = \frac{E\nu}{(1+\nu)(1-2\nu)}\:, \hspace{1em} \mu = \frac{E}{2(1+\nu)}\:.
\end{equation*}
Here, those constants are set to the classical academic values $E=1$ and $\nu=0.3$ and the penalty parameter $\varepsilon$ is set to $10^{-6}$. Such a value for $\varepsilon$ ensures that the solution $\uu_\varepsilon$ is close enough to the solution $\uu$ of the original contact problem.  

Concerning $D$ and the admissible shapes, at each iteration $k$, the current domain $\Omega^k$ will be contained in $D$ and its boundary $\partial\Omega^k$ will be divided as follows (the colours refer to \ref{fig:InitCanti2dCercle}):

\begin{itemize}[topsep=4pt, parsep=1pt, leftmargin=3em]
    \item $\Gamma_N^k=\Gamma_N^0$ is fixed as the orange part of $\partial D$,
    \item $\Gamma_D^k$ is the intersection of $\partial\Omega^k$ and $\hat{\Gamma}_D$, the blue part of $\partial D$, 
    \item $\Gamma_C^k$ is the transformation, through $\thetaa^k$, of the green part of $\partial D$.
\end{itemize}

Working with a formulation with no gap is quite convenient in several cases. First, when an a priori potential contact zone $\hat{\Gamma}_C\subset \partial D$ is known, then defining $\Gamma_C=\partial\Omega \cap \hat{\Gamma}_C$ enables to enforce the contact boundary to be part of $\partial D$. In such situations, see for example \cite{maury2017shape}, the boundary $\Gamma_C^k$ is either treated like $\Gamma_D^k$ (the contact area cannot be empty) or $\Gamma_N^k$ (the contact area is fixed) during the optimization process. Second, those formulations are well suited for interface problems involving several materials, see \cite{lawry2015level}. 

However, introducing a gap in the formulation allows to extend the method to situations where we want to optimize the shape of a body in contact with a rigid foundation (known a priori). Especially, the potential contact zone is included into the shape optimization process : as the contact zone is not fixed, the shape can be modified along $\Gamma_C$.

\subsection{The cantilever}

We revisit one of the most frequently presented test in shape optimization: the design of a bidimensional cantilever beam. This test differs from the usual one by the added possibility of a support of the beam through contact (sliding or frictional) with a rigid body. 
For this benchmark, the domain $D$ is the rectangular box $[0,2]\times [0,1]$ meshed with triangles, with an average number of vertices equal to 1300. The rigid foundation is the circle of radius $R=8$ and center $x_C=(1,-8)$. 
External forces are chosen such that $\ff=0$, and $\tauu=(0,-0.01)$ is applied on $\Gamma_N$ (in orange in \ref{fig:InitCanti2dCercle}). In the frictional case, $s=10^{-2}$ and $\mathfrak{F}=0.2$. And the weight coefficients in $J$ are $\alpha_1=15$, $\alpha_2=0.01$. These choices are based on the generic behavior of the model for the given data, and can be reinterpreted as searching for a stiff structure under volume constraint. Besides, since $\tauu=(0,-0.01)$, the order of magnitude of $\uu_\varepsilon$ is also $10^{-2}$, hence the difference between the orders of magnitude of $\alpha_1$ and $\alpha_2$.
\begin{figure}[h]
	\begin{center}
	\resizebox{0.65\textwidth}{!}{
    \begin{tikzpicture}

	\draw[>=latex,->] (-2,0) -- (-1.5,0);
	\draw[>=latex,->] (-2,0) -- (-2,0.5);
	\node[black] at (-1.25,0) {$x$};
 	\node[black] at (-2,0.75) {$y$};
    \node[white] at (10,0) {$\:$};
   
	\draw[black] (8.66,-0.62) arc (75:105:18);
	\draw[dartmouthgreen, very thick] (0,0) -- (8,0);
	\draw[black] (8,0) -- (8,1.6);
	\draw[orange, very thick] (8,1.6) -- (8,2.4);
	\draw[black] (8,2.4) -- (8,4);
	\draw[black] (8,4) -- (0,4);
	\draw[blue, very thick] (0,4) -- (0,0);

	\node[white] at (4,-0.3) {$\:$};
	\node[black] at (4,2) {\large{$D$}};
	\node[blue] at (0.7,2) {\large{$\hat{\Gamma}_D$}};
	\node[orange] at (7.3,2) {\large{$\Gamma_N$}};
	\node[dartmouthgreen] at (4,0.5) {\large{$\Gamma_C$}};
	\node[black] at (4,-0.5) {\large{$\Omega_{rig}$}};
	\end{tikzpicture}
    }	
	\end{center}
  	\caption{Initial geometry for the 2d cantilever.}
    \label{fig:InitCanti2dCercle}
\end{figure}
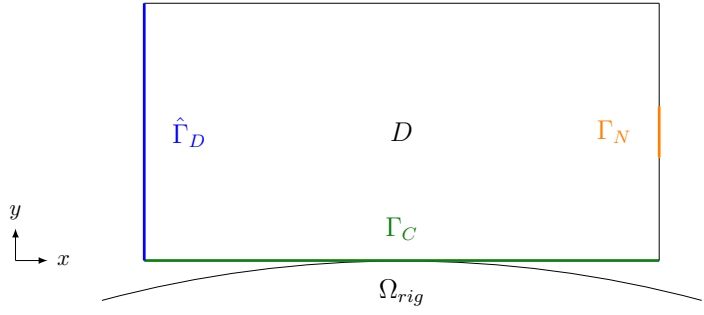

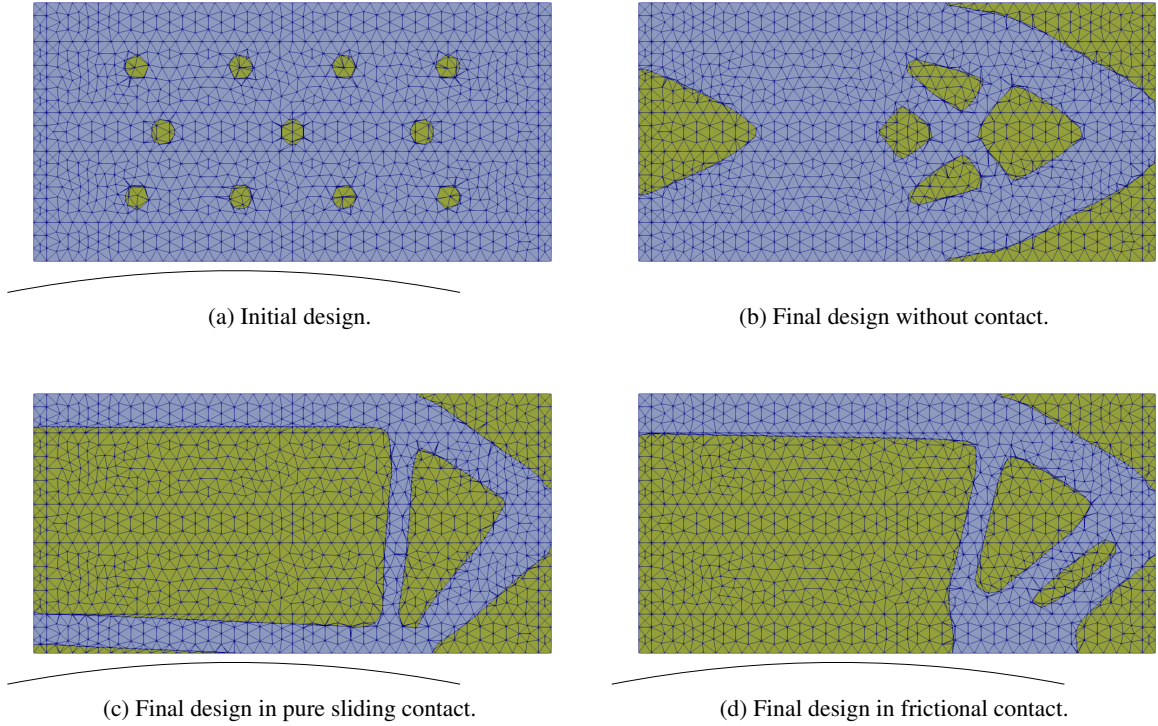
\begin{figure}[h]
  \begin{center}
    \subfloat[Initial design.]{
	\resizebox{0.46\textwidth}{!}{
	\begin{tikzpicture}
    	\node[anchor=south west,inner sep=0] at (0,0) {\includegraphics[width=\textwidth]{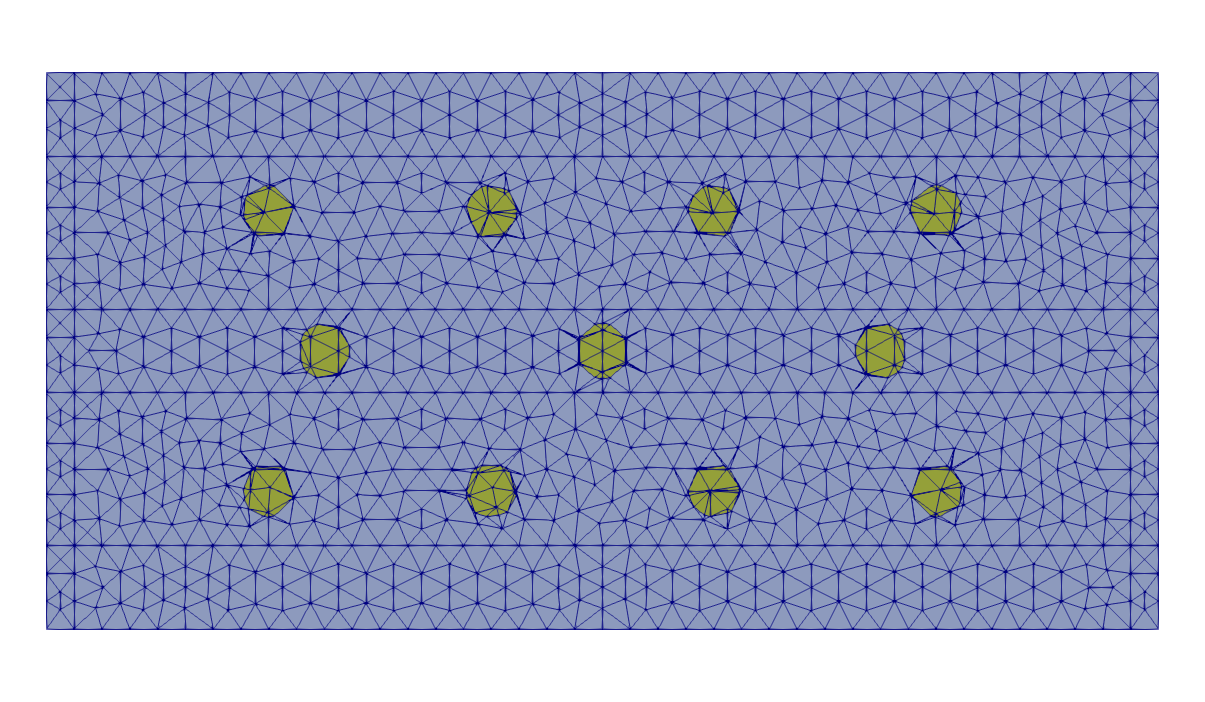}};
    	\draw[black] (13.2,0.08) arc (79:101:35);
    \end{tikzpicture}
    \label{sub:Canti2dCercleIt0}
    }
    }
    \hspace{.5em}
    \subfloat[Final design without contact.]{
	\resizebox{0.46\textwidth}{!}{
	\begin{tikzpicture}
    	\node[anchor=south west,inner sep=0] at (0,0) {\includegraphics[width=\textwidth]{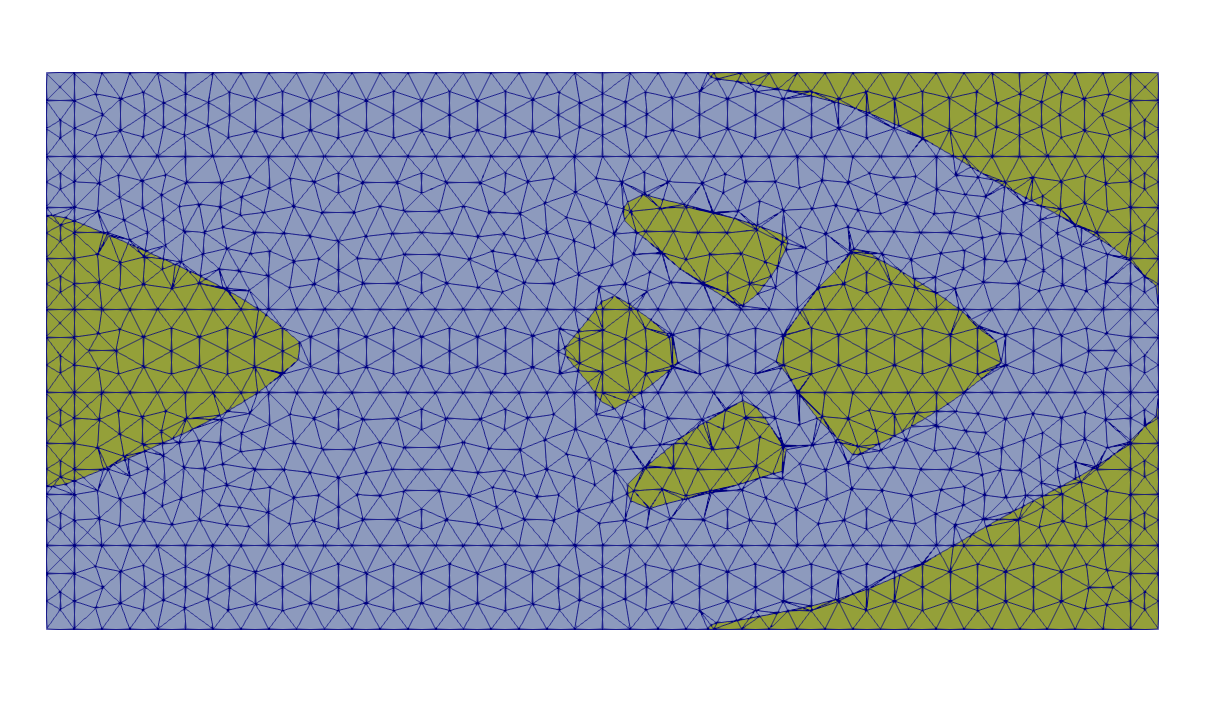}};
    	\draw[white] (13.2,0.08) arc (79:101:35);
    \end{tikzpicture}
    \label{sub:Canti2dBis}
    }
    }
    \hspace{.5em}
    \subfloat[Final design in pure sliding contact.]{
	\resizebox{0.46\textwidth}{!}{
	\begin{tikzpicture}
    	\node[anchor=south west,inner sep=0] at (0,0) {\includegraphics[width=\textwidth]{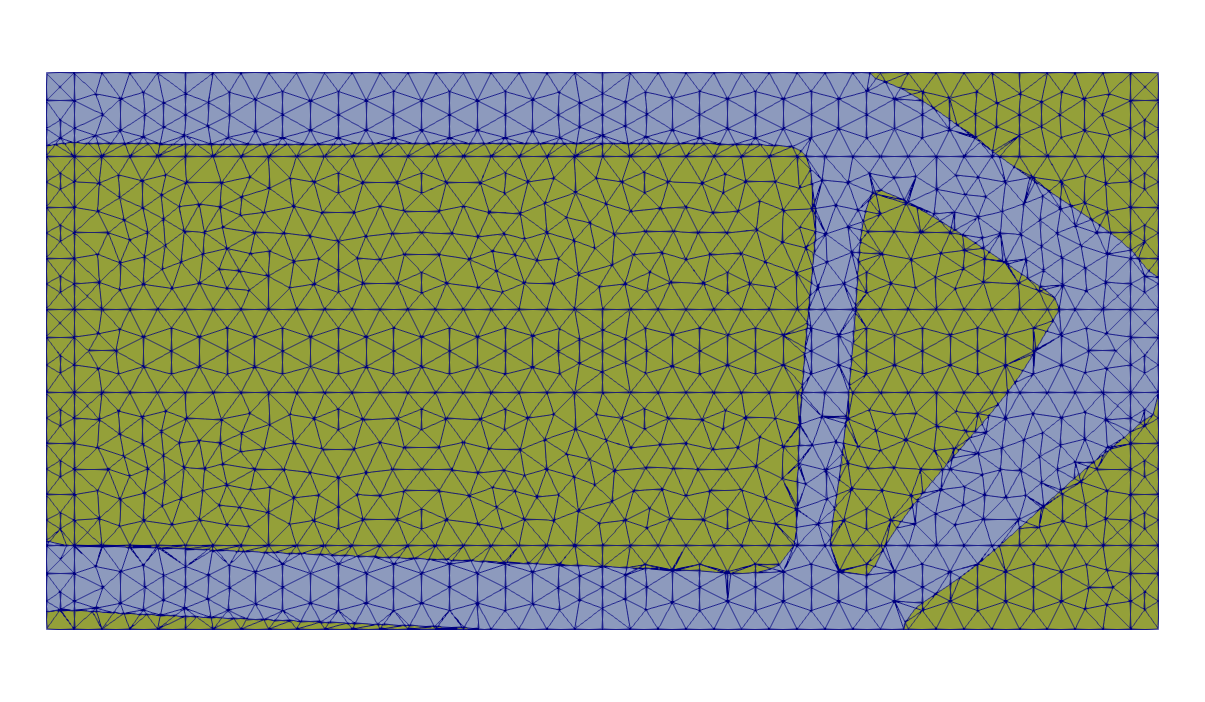}};
    	\draw[black] (13.2,0.08) arc (79:101:35);
    \end{tikzpicture}
    \label{sub:Canti2dCerclePena}
    }
    }
    \hspace{.5em}
    \subfloat[Final design in frictional contact.]{
	\resizebox{0.46\textwidth}{!}{
	\begin{tikzpicture}
    	\node[anchor=south west,inner sep=0] at (0,0) {\includegraphics[width=\textwidth]{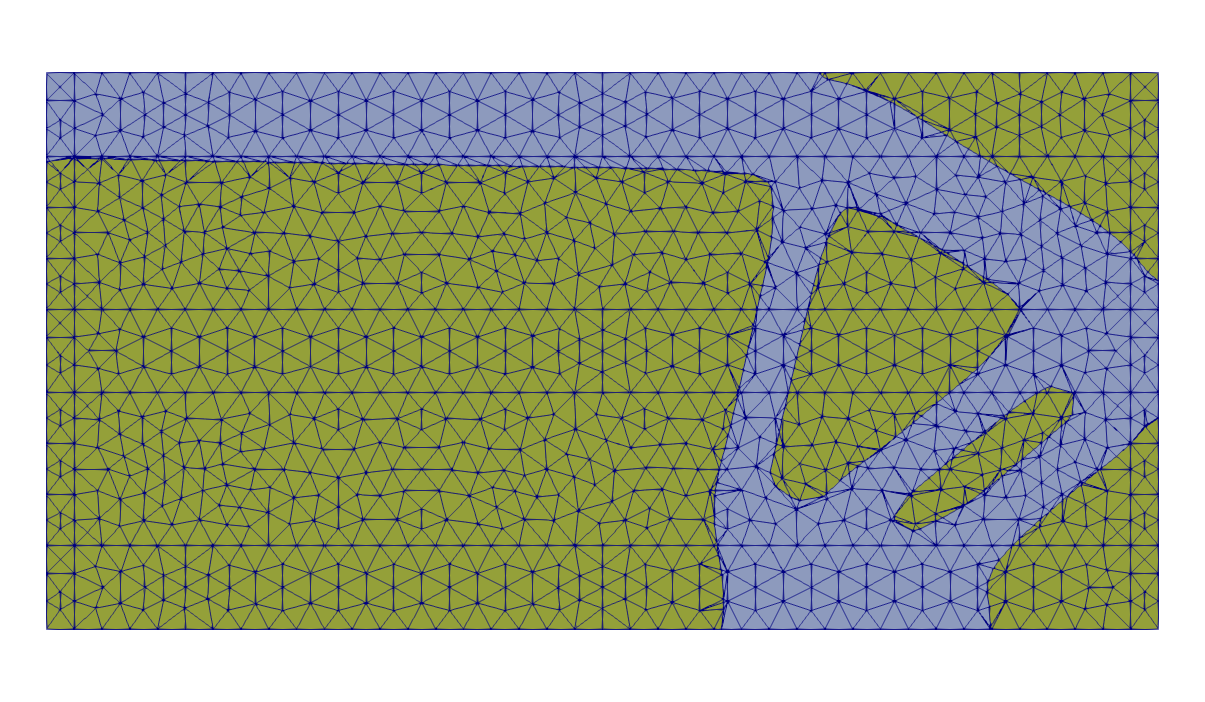}};
    	\draw[black] (13.2,0.08) arc (79:101:35);
    \end{tikzpicture}
    \label{sub:Canti2dCercleFrottPena}
    }
    }
    \end{center}
  \caption{Initial and final designs for the 2d cantilever in contact with a disk ($\Omega$ in blue, $D\setminus \Omega$ in yellow).}
  \label{fig:Canti2dCercle}
\end{figure}

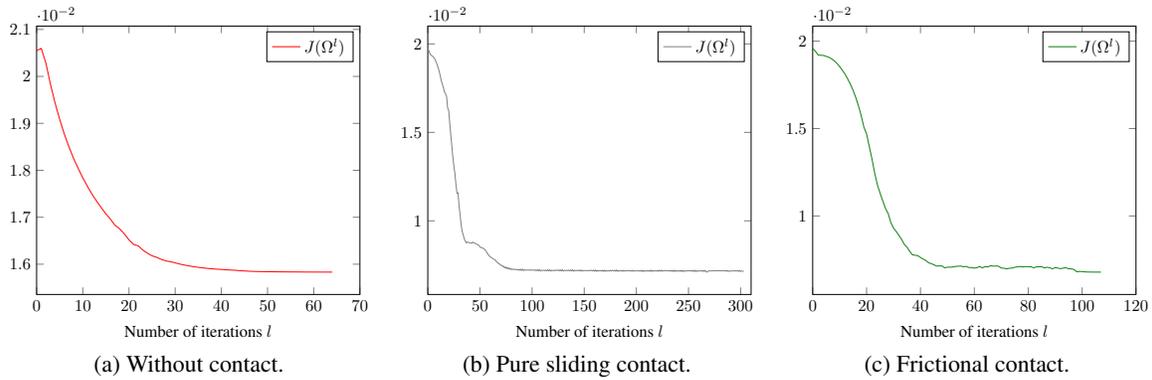
\begin{figure}
\begin{center}
    \subfloat[Without contact.]{   
	\resizebox{!}{0.27\textwidth}{
	\begin{tikzpicture}
		\begin{axis}[
    		xlabel={Number of iterations $l$},
    		xmin=0, xmax=70,
    		]
    		\addplot[color=red,mark=none] table {res_canti2d_bis.txt};
   			\legend{$J(\Omega^l)$}
   		\end{axis}
	\end{tikzpicture}
	}
	}
	%
    \subfloat[Pure sliding contact.]{   
	\resizebox{!}{0.27\textwidth}{
	\begin{tikzpicture}
		\begin{axis}[
    		xlabel={Number of iterations $l$},
    		xmin=0, xmax=310,
    		]
    		\addplot[color=gray,mark=none] table {res_canti2d_cercle_pena.txt};
   			\legend{$J(\Omega^l)$}
   		\end{axis}
	\end{tikzpicture}
	}
	}
	%
    \subfloat[Frictional contact.]{   
	\resizebox{!}{0.27\textwidth}{
	\begin{tikzpicture}
		\begin{axis}[
    		xlabel={Number of iterations $l$},
    		xmin=0, xmax=120,
    		]
    		\addplot[color=dartmouthgreen,mark=none] table {res_canti2d_cercle_frott_pena.txt};
   			\legend{$J(\Omega^l)$}
   		\end{axis}
	\end{tikzpicture}
	}
	}
\end{center}
\caption{Convergence history for the 2d cantilever.}
\label{fig:CvgCanti2d}
\end{figure}

We tested our algorithm on three different physical models: the standard elasticity model without contact (as if the disk was not here), the pure sliding model which does not take into account potential friction (i.e.$\!$ $\mathfrak{F}=0$), and the model of contact with Tresca friction.
In the case without contact, we recover the classical result, although the cantilever obtained might seem a little heavy due to our choice of coefficients $\alpha_1$ and $\alpha_2$. 
In the cases with contact, as expected, the optimal design suggested by the algorithm uses the contact with the rigid foundation as well as the clamped region to gain stiffness. More specifically, it seems that the effective contact zone (active set) has been moved to the right during the process. This makes sense because the closer the contact zone is to the zone where the load is applied, the stiffer will be the structure. 
However, in the frictional case, the tangential stress associated to friction phenomena $\sigmaa_{\normalInt\!\tanExt}$ points to the left and slightly upwards, since it is parallel to $\tanExt$ and opposed to the tangential displacement. This helps the structure to be stiffer as it compensates part of the downward motion induced by the traction $\tauu$, which might explain why the optimal shape requires only one anchor point instead of two for the pure sliding case.

The convergence history for all cases is displayed figure \ref{fig:CvgCanti2d}. Note that the convergence is much faster in the case without contact, which was predictable since the mechanical problem is easier to solve, thus the shape derivatives should be more accurate. Moreover, the final value of $J$ is around $1.6$ in the case without contact whereas it is around $0.7$ in both cases with contact. Indeed, due to the possibility of laying onto a rigid foundation, the models with contact lead to better designs.

\section{Conclusion}

In this work, we expressed conditions (similar to strict complementarity conditions) that ensure shape differentiability of the solution to the penalty formulation of the contact problem with prescribed friction based on a Tresca model. In order to achieve this goal, we relied on Gâteaux differentiability, combined with an assumption on the measure of some subsets of the contact region where non-differentiabilities may occur. Under such assumptions, we derived an expression for the shape derivative of any general functional. 
Finally, this expression has been used in a gradient descent algorithm, which we tested on a revisited version of the classical cantilever benchmark. 

As far as future work is concerned, 
the idea of working with directional shape derivatives could be extended to other formulations where non-Gateaux-differentiable operators are involved: e.g.$\!$ the Augmented Lagrangian formulation or Nitsche-based formulations.

\bibliography{references,bibliography}

\end{document}